\documentclass[reqno,twoside]{amsart}
\usepackage{amsfonts,xcolor,eucal}
\usepackage{a4wide,amsmath,amssymb,latexsym,amsthm}
\usepackage[dvips,bottom=1.4in,right=1in,top=1in, left=1in]{geometry}
\usepackage{graphicx,epstopdf,xspace,enumerate,pstricks}
\usepackage{mathrsfs,mathtools,epic,bm,cancel}

\usepackage[colorlinks=true,urlcolor=blue,
citecolor=red,linkcolor=blue,linktocpage,pdfpagelabels,
bookmarksnumbered,bookmarksopen]{hyperref}
\usepackage[english]{babel}

\newtheorem{theorem}{Theorem}[section]
\newtheorem{proposition}[theorem]{Proposition}

\theoremstyle{definition}
\newtheorem{definition}[theorem]{Definition}

\newtheorem{remark}[theorem]{Remark}

\hypersetup{urlcolor=blue, citecolor=red}

\newfont{\script}{eusm10 scaled\magstep1}

\newcommand{\N}{\mathbb{N}}

\newcommand{\R}{\mathbb{R}}

\newcommand{\dt}{\mathrm{d}t}
\newcommand{\dx}{\mathrm{d}x}
\newcommand{\dy}{\mathrm{d}y}

\def\Omc{\mathbb{R}^N\setminus\Omega}

\def\Omb{\mathbb{R}^N\setminus\overline{\Omega}}

\def\RR{{\mathbb{R}}}
\def\NN{{\mathbb{N}}}

\def\bOm{\overline{\Om}}
\def\pOm{\partial \Omega}

\def\Om{\Omega}
\def \dis {\displaystyle}

\numberwithin{equation}{section}

\usepackage{xcolor}
\usepackage{soul}

 \usepackage[pagewise]{lineno}  

\title[Mixed local-nonlocal parabolic PDE]{Optimal control of mixed local-nonlocal parabolic PDE with singular boundary-exterior data}

\author[Djida]{Jean-Daniel Djida}
\author[Mophou]{Gis\`{e}le Mophou}
\author[Warma]{Mahamadi Warma}

\address[Jean-Daniel Djida]{African Institute for Mathematical Sciences (AIMS), P.O. Box 608, Limbe Crystal Gardens,South West Region, Cameroon.}
\email[Djida]{jeandaniel.djida@aims-cameroon.org}

\address[Gis\`{e}le Mophou]{Laboratoire L.A.M.I.A., D\'{e}partement de Math\'{e}matiques et Informatique, Universit\'{e} des Antilles, Campus Fouillole, 97159 Pointe-\`{a}-Pitre,(FWI), Guadeloupe.}
\email[Mophou]{gisele.mophou@univ-antilles.fr}

\address[Mahamadi Warma]{Department of Mathematical Sciences and the Center for Mathematics and Artificial Intelligence (CMAI),
George Mason University, Fairfax, VA 22030, USA.}
\email[Warma]{mwarma@gmu.edu}

\thanks{The first author is supported by the Deutscher Akademischer Austausch Dienst/German Academic Exchange Service (DAAD). The third author is partially supported by the AFOSR under Award NO:  FA9550-18-1-0242 and by the US Army Research Office (ARO) under Award NO: W911NF-20-1-0115.}



\begin{document}

\begin{abstract}
We consider parabolic equations on bounded smooth  open sets $\Om\subset \R^N$ ($N\ge 1$) with mixed Dirichlet type boundary-exterior conditions associated with the elliptic operator $\mathscr{L} \coloneqq - \Delta + (-\Delta)^{s}$ ($0<s<1$). Firstly, we prove several well-posedness and regularity results of the associated elliptic and parabolic problems with smooth, and then with singular boundary-exterior data.
Secondly, we show the existence of optimal solutions of associated optimal control problems, and we characterize the  optimality conditions. This is the first time that such topics have been presented and studied in a unified fashion for mixed local-nonlocal PDEs with singular data.
\end{abstract}

\maketitle

\section{Introduction}


Let $\Omega\subset\R^N$ ($N\ge 1$) be a bounded domain with a smooth boundary $\pOm$.
We consider
the minimization problem:
\begin{subequations}\label{Eq:main_A}
\begin{equation}\label{Eq:func_A1}
\min_{(u_1,u_2) \in \mathcal{Z}_{D}}J(\psi(u_1,u_2)) ,
\end{equation}
subject to the constraints that the state  $\psi := \psi(u_1,u_2)$ solves the following initial-boundary-exterior value problem:
\begin{align}\label{Eq:main_A1}
\begin{cases}
\psi_{t} + \mathscr{L}\psi = 0 & \mbox{ in }\; Q \coloneqq \Omega\times(0,T),\\
\psi=u_1&\mbox{ on }\;\Gamma:=\pOm\times(0,T),\\
\psi = u_2 &\mbox{ in }\; \Sigma \coloneqq (\Omc)\times (0,T), \\
\psi(\cdot,0) = 0, &\mbox{ in }\; \Omega.
\end{cases}
\end{align}
Here, the operator $\mathscr{L}$ is given by
\begin{equation}\label{LOPERATO1}
\mathscr{L} \coloneqq - \Delta + (-\Delta)^{s},\qquad 0<s<1,
\end{equation}
the functional $J: \mathcal{Z}_{D}\to [0,\infty]$ is weakly lower-semicontinuous (we shall give the precise expression of $J$ later),
the control $(u_1,u_2) \in \mathcal{Z}_{ad}$ with $\mathcal{Z}_{ad} \subset \mathcal{Z}_{D}$ being a closed and convex subset, where
\begin{equation}\label{Eq:Control_constA}
  \mathcal{Z}_{D} \coloneqq L^2(\Gamma)\times L^{2}(\Sigma).
\end{equation}
\end{subequations}

In \eqref{LOPERATO1}, $\Delta$ is the classical Laplacian and $(-\Delta)^s$ ($0<s<1$) denotes the fractional Laplace operator given formally by the following singular integral:
\begin{align*}
(-\Delta)^s\psi=\mbox{P.V.}\;C_{N,s}\int_{\RR^N}\frac{\psi(x)-\psi(y)}{|x-y|^{N+2s}}\;dy,
\end{align*}
where $C_{N,s}$ is a normalization constant depending only on $N$ and $s$. We refer to Section \ref{Sec:preli} for more details.

Let us clarify how we interpret the boundary-exterior conditions in \eqref{Eq:main_A1}.
\begin{itemize}
\item If the function $\psi$  has a well-defined trace on $\Gamma$, then $u_1=\psi|_{\Gamma}$. In that case, the condition on $\Gamma$ can be dropped and the problem will still be well-posed as we shall see later. But if the condition in $\Sigma$ is removed,  then the system will be ill-posed.

\item If $\psi$ does not have a well-defined trace on $\Gamma$, then  the condition on $\Gamma$ will be seen in a very-weak sense that we shall explain later. In that case, none of the two conditions (boundary and exterior)  can be removed, otherwise the system will be ill-posed.
\end{itemize}
These important facts will be clarified in Section \ref{Sec:Wellpo}.\medskip

The first main concern of the present paper is to prove several well-posedness results of the parabolic problem \eqref{Eq:main_A1} and the  associated elliptic (time independent) equation:
\begin{equation}\label{E-E}
\begin{cases}
 \mathscr{L}\phi = f & \mbox{ in }\;  \Omega,\\
\phi=u_1&\mbox{ on }\; \pOm,\\
\phi = u_2 &\mbox{ in }\; \Omc.
\end{cases}
\end{equation}
Notice that throughout the paper, since we are considering smooth open sets $\Omega$, it follows that a.e. in $\Omc$ is the same as a.e. in $\Omb$.
The system \eqref{E-E}, with $f\in L^2(\Omega)$ and $u_1,u_2$ are zero, has been very recently studied in \cite{BDVV2020} where the authors have proved some well-posedness,  local and boundary regularity, and some maximum principle results.  Here, we shall show that in this case, the associated self-adjoint operator on $L^2(\Om)$ is a generator of a strongly continuous submarkovian semigroup $(T(t))_{t\ge 0}$ which is also ultracontractive in the sense that the operator $T(t)$ maps $L^1(\Omega)$ into $L^\infty(\Om)$ for every $t>0$. This will be used to have some fine regularity results of the dual problem associated with the system \eqref{Eq:main_A1} which will be crucial in the study of our optimal control problems.\medskip

In \cite{BDVV2020},  the authors also considered briefly the case where $u_1$ and $u_2$ in \eqref{E-E} are smooth functions. In that case, as we mentioned above, the condition on $\partial\Omega$ can be dropped and the associated system will still be well-posed.
In the first part of the present article, for non-smooth boundary-exterior data, we shall introduce the notion of solutions by transposition (or very-weak solutions) of  \eqref{E-E}, study their existence and regularity.  Our main result in this direction reads that if $u_1\in L^2(\pOm)$, $u_2\in L^2(\Omc)$ and $0<s\le 3/4$, then the associated very-weak solution $\psi$ of  \eqref{E-E} belongs to $H^{1/2}(\Omega)\cap L^2(\R^N)$ (see Theorem \ref{th-36} and Remark \ref{rem-37}).  For the associated parabolic problem \eqref{Eq:main_A1}, if $u_1\in L^2(\Gamma)$, $u_2\in L^2(\Sigma)$ and $0<s\le 3/4$, then $\psi\in L^2((0,T)\times \R^N)\cap C([0,T];H^{-1}(\Omega))$ (see Theorem \ref{thm:VWS_Exist_1}).\medskip

Notice that the eigenvalues problem associated with nonlocal Neumann exterior conditions, that is, when $\phi=0$ on $\pOm$ and $\phi=0$ in $\Omc$ are replaced with $\partial_\nu\phi=0$ on $\pOm$ and $\mathcal N_s\phi=0$ in $\Omb$,  respectively (see \eqref{NonlocalDeri} below for the definition of $\mathcal N_s)$ has been recently investigated in \cite{DLV}, where the authors have shown that the associated operator has a compact resolvent, hence, has  a discrete spectrum formed with eigenvalues.  We mention that the case of the fractional Laplace operator with the nonlocal exterior condition $\mathcal N_s\phi=0$ in $\Omb$ has been introduced and investigated in \cite{DRV}.
Here, we are not interested in the Neumann type boundary-exterior conditions.

Our second main concern is to study the existence of optimal solutions to optimal control problems involving the mixed operator $ \mathscr{L}$ with singular Dirichlet boundary-exterior data,  and to characterize the associated optimality conditions. More precisely, we shall consider the following two different optimal control problems:
\begin{equation}\label{OPC}
\min_{(u_1,u_2)\in \mathcal{Z}_{ad}} J_i((u_1,u_2)),\; i=1,2,
\end{equation}
subject to the constraint that the state $\psi\coloneqq \psi(u_1,u_2)$ solves the parabolic system \eqref{Eq:main_A1}. We recall that the control $(u_1,u_2) \in \mathcal{Z}_{ad}$ with $\mathcal{Z}_{ad}$ being a closed and convex subset of $\mathcal{Z}_{D}\coloneqq  L^2(\Gamma)\times L^{2}(\Sigma)$, which is endowed with the norm given by
\begin{equation*}
\|(u_{1},u_{2})\|_{ \mathcal{Z}_{D}}=\Big(\|u_{1}\|^2_{L^2(\Gamma)}+\|u_{2}\|^2_{ L^{2}(\Sigma)}\Big)^{\frac 12}.
\end{equation*}
The functionals $J_1$ and $J_2$ are given by
\begin{equation}\label{dJ1}
J_1(u_1,u_2)\coloneqq \frac{1}{2}\|\psi((u_{1},u_{2}))-z_d^1\|_{L^2(Q)}^2 + \frac{\beta}{2}\|(u_{1},u_{2})\|^{2}_{\mathcal{Z} _{D}}
\end{equation}
and
\begin{equation}\label{dJ2}
J_2(u_1,u_2)\coloneqq \frac{1}{2}\|\psi(T;(u_{1},u_{2}))-z_d^2\|_{H^{-1}(\Om)}^2 + \frac{\beta}{2}\|(u_{1},u_{2})\|^{2}_{\mathcal{Z} _{D}},
\end{equation}
where  $\beta>0$ is a real number,  $z_d^1\in L^2(Q) $,  $z_d^2\in H^{-1}(\Om)$,
and
\begin{align*}
\|\phi\|^2_{H^{-1}(\Omega)}:=\langle (-\Delta_D)^{-1}\phi,\phi\rangle_{H^1_0(\Omega),H^{-1}(\Omega)}.
\end{align*}
Here, $-\Delta_D$ is the realization in $L^2(\Om)$ of the Laplace operator $-\Delta$ with the zero Dirichlet boundary condition.  The functionals $J_1$ and $J_2$ can be replaced with more general functionals satisfying suitable conditions, without any substantial modification of the proofs. \medskip

The novelties and difficulties of the present paper can be summarized as follows.

\begin{enumerate}
\item For the first time, elliptic and parabolic equations associated with a mixed local-nonlocal operator and singular boundary-exterior data have been studied.

\item Since we are considering singular data, the right definition of solutions, their existence, and their fine regularity, are more challenging
 than the case of the single local, or the single nonlocal operator.  In particular, the regularity of solutions is crucial in the study of the optimal control problems.
 
 \item For singular boundary-exterior data, we have introduced the notion of solutions by transposition to the system \eqref{Eq:main_A1}.  This definition requires that solutions $\phi$ of the dual system associated with \eqref{Eq:main_A1}  satisfy $\partial_\nu \phi\in L^2(\Gamma)$ and $\mathcal N_s\phi\in L^2(\Sigma)$, where $\partial_\nu \phi$ is the classical normal derivative of $\phi$, and  $\mathcal N_s\phi$ denotes the nonlocal normal derivative of $\phi$ (see \eqref{NonlocalDeri}). We have been able to show this regularity only in the range of exponents $0<s\le 3/4$. The case $3/4<s<1$ remains an open problem.   In the classical local case $s=1$ or $s=0$ (in the situation of the present paper),  classical elliptic regularity results show that $\phi$ belongs to $L^2((0,T);H^2(\Omega))$ so that its normal derivative exists and belongs to $L^2((0,T);L^2(\pOm))$.  This seems not to be the case for the nonlocal case investigated here.  We shall give more details in Section \ref{Sec:Wellpo}.
 
\item  Let us notice that the operator $\mathscr{L}$ is a sum of the local operator $-\Delta$ and the nonlocal one $(-\Delta)^s$.  On the one hand, for this operator, regarding well-posedness and regularity of associated elliptic and parabolic systems, the nonlocal operator seems to be dominant as we shall see in Section \ref{Sec:Wellpo}. This shows that the operator $\mathscr{L}$ cannot be just seen as a simple perturbation of $-\Delta$ with a lower order operator. To see that, for example,  $\mathscr{L}\psi\in L^2(\Om)$ does not mean that $\Delta \psi\in L^2(\Omega)$ and $(-\Delta)^s\psi\in L^2(\Omega)$. This is the case only for certain values of $s$, that is, when $0<s\le 3/4$. 
On the other hand,  regarding the ultracontractivity property of the semigroup $(T(t))_{t\ge 0}$ mentioned above, the Laplace operator $-\Delta$ seems to be dominant.

 \item For the first time, optimal control problems associated with a mixed local-nonlocal operator and singular boundary-exterior data have been investigated.  The existence and uniqueness of minimizers and the characterization of the associated optimality conditions have been obtained under the assumption that $0<s\le 3/4$.
 This is a substantial extension of the nowadays well-known local case of the Laplace operator (see e.g.  the monograph \cite{lions1971} and the references therein), and the nonlocal case recently investigated in \cite{AKW,AVW1,AVW2,AntilWarma} and the references therein.
\end{enumerate}

It is a great idea from the authors in \cite{BDVV2020} for having introduced the operator $\mathscr{L}$ and initiated the study of its qualitative properties.

We mention that fractional order operators have recently received a great deal of attention due to the fact that they model many diverse real-life phenomena which could not be adequately
modeled by using classical local or non-fractional operators. These type of operators belong to the broad class of nonlocal operators and have merged as a modeling alternative in various branches of science. The theory of these operators has discovered many applications, for example in 
fluid dynamics \cite{Chen}, diffusion of biological species \cite{Viswanathan}, transitions across an interface,  multiscale behavior in cardiac tissue \cite{Bueno}, and phase field models \cite{Antil}.
Despite the numerous applications, still, there is a lot of undiscovered potential lying in the connection between local (induced by Brownian motion) and nonlocal (induced by L\'{e}vy process), since they can both be understood from the viewpoint of mathematical analysis and then combined in a natural way. However, there is no agreement on exactly how to define mixed local and nonlocal operators as a combination process of the Brownian and L\'{e}vy processes. For this reason, attempts have been made recently in \cite{BDVV2020}, to unify local and nonlocal operators in the single operator $\mathscr{L}$ given in \eqref{LOPERATO1}.
Operators of the kind \eqref{LOPERATO1} arise naturally from the superposition of two stochastic processes with different scales (namely, a classical random walk and a L\'evy flight): roughly speaking, when a particle can follow either of these two processes according to a certain probability, the associated limit diffusion equation is described by an operator of the form \eqref{LOPERATO1}. See also \cite{Cantrell} which discuss about the advection-mediated coexistence of competing species.\medskip

Optimal control problems of partial differential equations  involving classical local operators with boundary control through  Dirichlet and Neumann conditions have been widely investigated this last five decades. For instance, J. L. Lions \cite{lions1971} considered an optimal control problem with Neuman boundary observation subjected to an elliptic equation with Dirichlet boundary control. Using a transposition method,  the author proved the existence and uniqueness of solutions when the control is in  $L^2(\partial \Omega)$ and gave a sense to the normal derivative of the considered  observation in $H^{-1}(\partial \Omega)$. Then,  the author showed the existence of the optimal control and gave its characterization. These results have been  extended by the same author to an optimal control problem with Neuman boundary observation subjected to a parabolic equation with Dirichlet boundary control  in \cite{lions1971}.
Note that for both elliptic and parabolic equations, if the set of admissible controls are respectively $L^2(\partial \Omega)$ and $L^2((0,T)\times \partial \Omega)$,  and the domain $\Omega$ is smooth enough, then the regularity of the control and the  normal derivative of the adjoint state are the same. 

We observe that there is a fundamental difference between the regularity of solutions of parabolic problems with Dirichlet $L^2$-boundary data and the regularity of  optimal solutions of the associated control problems.  In the classical local case ($s=1$ or $s=0$ in the setting of the present paper),   it turns out that the latter is always higher.
In fact, Lions and Magenes \cite[Section 2.2]{Lions_682} were first to prove that solutions of the open loop system with $L^2(\Gamma)$-Dirichlet data belong to $L^2((0,T);H^{1/4}(\Omega))\cap H^{1/2}((0,T);L^2(\Omega))$.
In \cite{Lasiecka_78},   Lasiecka used a semigroup  approach to solve a  Dirichlet boundary optimal control and a distributed observation subjected to  parabolic systems. The author proved that if $\Omega$ is smooth enough or $\Omega$ is a parallelepiped,  then the optimal control belongs  to  $L^2((0,T); H^{1/2}( \partial \Omega))$ which is more regular that the allowed control space, $L^2(\Gamma)$,  and the optimal trajectory is in $L^2((0,T); H^1(\Omega))$.  These results were significantly improved by Lasiecka and Triggiani \cite{Lasiecka_83,LT-20} (see also \cite{Sorine_82}) for smooth domains.  Actually,  considering a more general parabolic problem with a control in $L^2(\Gamma)$ acting on a Dirichlet  boundary,  and a distributed observation,  using a transposition method, the authors proved  that  the optimal control  belongs to $L^2((0,T);H^{1/2}( \partial \Omega)) \cap H^{1/4}((0,T); L^2( \partial \Omega))$ and the optimal trajectory is in $L^2((0,T); H^1(\Omega)) \cap H^{1/2}((0,T); L^2( \partial \Omega))$.  An extension to non-autonomous systems is contained in \cite{AFT}.
Whether such a result can be obtained in the fractional case considered in the present paper is still an interesting open problem.  Notice that for the fractional case,  we have an exterior control that is no longer a trace of the state on $\Gamma:=\pOm\times (0,T)$, but a restriction of the state in $(\mathbb R^N\setminus \Om)\times(0,T)$.  
For more information on boundary optimal control problems,  we refer to \cite{Lions_681,Lions_682,lions1971,Fredi} and the references therein. 

Recently,  partial differential equations with time fractional derivatives and/or the Fractional Laplace operator have emerged as excellent tools to described phenomenon with memory effects. Hence, the control of such equations are  of great interest.  However, for time fractional diffusion equations involving classical second order elliptic operators,  boundary conditions are the same as for the classical diffusion equations.  Instead,  for diffusion equations involving the fractional  Laplace operator,  they may not have boundary conditions but external conditions. These lead to boundary control problems for time fractional diffusion equations involving second order elliptic operators, and external control problems for diffusion equations involving the fractional  Laplace operator. Compared to the classical  diffusion equations,  the literature on  quadratic boundary control associated to such equations  is scarce.  In \cite{Dorville_2011},  R. Dorville  et al. used a transposition method to study a Dirichlet boundary control problem associated to a time fractional diffusion equation involving Riemann-Liouville fractional derivatives of order $\alpha \in (0,1)$ and a final time observation of Riemann-Liouville integrals of order $\alpha \in (0,1)$. They succeeded in  proving the existence and the uniqueness of the optimal control that they characterized by an optimality system. They showed  that the optimal control is in  $L^2((0,T);H^{1/2}( \partial \Omega))$ when the domain is smooth enough and the set of the admissible controls is $L^2((0,T)\times \partial \Omega)$.  

Using variational methods,  the authors in \cite{AKW} investigated an external optimal control problem of a nonlocal  elliptic equation involving the fractional Laplace operator with distributed observation. They proved the existence and uniqueness of the optimal control and gave the  optimality system that characterized this optimal control. They also noticed that even if the domain $\Omega$ is smooth enough,  contrary to the results obtained for the classical Laplace operator, the external optimal control remains in  $L^2((0,T)\times \R^N\setminus (\Omega))$, the allowed control space. This results was extended   in \cite{AVW1} to the  external optimal control problem of parabolic equations involving the fractional Laplace operator with distributed observation. 

In this paper, we consider two boundary quadratic  optimal control problems  subject to a diffusion equation involving the classical and the fractional Laplace operators. We have two controls: an external control and a control  acting through a Dirichlet condition. We first consider  distributed observations and prove the existence and uniqueness of the optimal control. Then,  observing as in \cite{Lasiecka_83} that for an allowed control in $L^2((0,T)\times \partial \Omega)$, it may happen that the response $\psi$ is such that $\psi(\cdot,T)\notin L^2(\Omega)$, we consider for the second control problem a final time observation in $H^{-1}(\Omega)$ and prove the existence and uniqueness of the optimal control.  In both optimal control problems we observe that if the set of admissible controls is $L^2(\Gamma)\times L^2(\Sigma) $,  then even though, we deal with a mixed local and nonlocal diffusion operator, optimal control belongs to $L^2((0,T);H^{1/2}( \partial \Omega))\times L^2((0,T); H^{1}_{\rm loc}(  \R^N\setminus \Omega))$ and the  optimal trajectory remains in $L^2((0,T)\times \R^N)$ whenever $0<s\le 3/4$.  We have not been successful in the case $0< s<3/4$.  Of course the cases $s=0$ and $s=1$ correspond to the classical Laplace operator mentioned above.\\

The rest of the paper is structured as follows. In Section \ref{Sec:preli} we fix some notations, give a rigorous definition of the fractional Laplace operator, and introduce the function spaces needed to study our problems. The results of well-posedness and regularity of solutions to the elliptic problem \eqref{E-E}, and the parabolic problem \eqref{Eq:main_A1}, are contained in Sections \ref{sec-31} and \ref{sec-32}, respectively.  In Section \ref{First-OCP} we show the existence and uniqueness of optimal solutions to the control problem \eqref{OPC}-\eqref{dJ1}, and we characterize the associated optimality conditions. The same study for the control problem  \eqref{OPC}-\eqref{dJ2} is contained in Section \ref{sec-OCP}.

\section{Notations and Preliminaries}\label{Sec:preli}

In this section we fix some notations and recall some known results as they are needed throughout the paper. These results can be found for example in  \cite{AVW1,AntilWarma,BDVV2020,BBC,Caf1,NPV,GW,GW-CPDE,GW2,Gris,War,War-In} and the references therein.  \medskip

Let us first give a rigorous definition of the fractional Laplacian. Given $0<s<1$,  we let
\begin{align*}
\mathcal L_s^{1}(\RR^N)\coloneqq \left\{w:\RR^N\to\RR\;\mbox{ measurable and}\; \int_{\RR^N}\frac{|w(x)|}{(1+|x|)^{N+2s}}\;\dx<\infty\right\}.
\end{align*}
For $w\in \mathcal{L}_s^{1}(\RR^N)$ and $\varepsilon>0$, we set
\begin{align*}
(-\Delta)_\varepsilon^s w(x)\coloneqq C_{N,s}\int_{\{y\in\RR^N:\;|x-y|>\varepsilon\}}\frac{w(x)- w(y)}{|x-y|^{N+2s}}\;\dy,\;\;x\in\RR^N,
\end{align*}
where $C_{N,s}$ is a normalization constant given by
\begin{align}\label{CNs}
C_{N,s}\coloneqq \frac{s2^{2s}\Gamma\left(\frac{2s+N}{2}\right)}{\pi^{\frac{N}{2}}\Gamma(1-s)}.
\end{align}
The fractional Laplacian $(-\Delta)^s$ is defined by the following singular integral:
\begin{align}\label{fl_def}
(-\Delta)^sw(x)\coloneqq C_{N,s}\,\mbox{P.V.}\int_{\RR^N}\frac{w(x)-w(y)}{|x-y|^{N+2s}}\;\dy =
\lim_{\varepsilon\downarrow 0}(-\Delta)_\varepsilon^s w(x),\;\;x\in\RR^N,
\end{align}
provided that the limit exists for a.e. $x\in\RR^N$.  We refer to \cite{NPV} and their references regarding the class of functions for which the  limit in \eqref{fl_def} exists for a.e. $x\in\RR^N$.\medskip

Next, we introduce the function spaces needed to study our problems. We start with fractional order Sobolev spaces.

Let $\Omega\subset\RR^N$ be an arbitrary open set. Given $0<s<1$ a real number,  we let
\begin{align*}
H^{s}(\Omega)\coloneqq \left\{w \in L^2(\Omega):\;\int_{\Omega}\int_{\Omega}\frac{|w(x)-w(y)|^2}{|x-y|^{N+2s}}\;\dx\dy<\infty\right\},
\end{align*}
and we endow it with the norm defined by
\begin{align*}
\|w\|_{H^{s}(\Omega)}\coloneqq\left(\int_{\Omega}|w(x)|^2\;\dx+
\int_{\Omega}\int_{\Omega}\frac{|w(x)-w(y)|^2}{|x-y|^{N+2s}}\;\dx\dy\right)^{\frac 12}.
\end{align*}
We set
\begin{align*}
H_0^{s}(\Omega)\coloneqq\Big\{w\in H^{s}(\RR^N):\;w=0\;\mbox{ in }\;\RR^N\setminus\Omega\Big\}.
\end{align*}
Then, $H_0^{s}(\Omega)$ endowed with the norm
\begin{equation}\label{norm0}
\|w\|_{H_0^{s}(\Omega)}=\left(\int_{\R^N}\int_{\R^{N}}
\frac{|w(x)-w(y)|^2}{|x-y|^{N+2s}}\;\dx\, \dy\right)^{1/2},
\end{equation}
is a Hilbert space (see e.g. \cite[Lemma 7]{servadei}).  
We let $H^{-s}(\Omega)\coloneqq (H_0^s(\Omega))^\star$ be the dual space of $H_0^s(\Omega)$ with respect to the pivot space $L^2(\Om)$, so that we have the following continuous and dense embeddings (see e.g. \cite{ATW}):
\begin{equation}\label{injection1}
 H_0^{s}(\Omega)\hookrightarrow L^2(\Omega)\hookrightarrow H^{-s}(\Omega).
 \end{equation}
Under the assumption that $\Omega$ is bounded and has a Lipschitz continuous boundary,  we have that $\mathcal D(\Om)$ is dense in $H_0^s(\Om)$ (for every $0<s<1$,  see e.g. \cite{FSV}), and by \cite[Chapter 1]{Gris}, if $0<s\ne 1/2<1$, then
 \begin{align*}
 H_0^s(\Omega)=\overline{\mathcal D(\Omega)}^{H^s(\Omega)},
 \end{align*}
with equivalent norms, where $\mathcal D(\Omega)$ denotes the space of all continuously infinitely differentiable functions with compact support in $\Omega$. But if $s=1/2$, then $H_0^s(\Omega)$ is a proper subspace of $\overline{\mathcal D(\Omega)}^{H^s(\Omega)}$.

For more information on fractional order Sobolev spaces, we refer to \cite{NPV,Gris,War} and their references.\medskip

Next, for $\varphi\in H^{s}(\RR^N)$ we introduce the {\em nonlocal normal derivative $\mathcal N_s$} given by
\begin{align}\label{NonlocalDeri}
\mathcal N_{s}\varphi(x)\coloneqq C_{N,s}\int_{\Omega}\frac{\varphi(x)-\varphi(y)}{|x-y|^{N+2s}}\; \dy,~~~~x\in\RR^N\setminus\bOm,
\end{align}
where $C_{N,s}$ is the constant given in \eqref{CNs}.   We notice that the nonlocal normal derivative $\mathcal N_s$ has been first introduced in \cite{DRV}.

\begin{remark}\label{rem-21}
It has been shown in \cite{GRSU,GSU} that the operator $\mathcal N_s$ maps $H^s(\RR^N)$ into $H_{\rm loc}^s(\Omb)$
Furthermore, if $u\in H_0^1(\Omega)$ and $(-\Delta)^s u\in L^2(\Omega)$, then $\mathcal N_su\in L^2(\Omb)$, and there is a constant $C>0$ such that
\begin{align}\label{NND}
\|\mathcal N_su\|_{L^2(\Omb)}\le C\|u\|_{H_0^1(\Omega)}.
\end{align}
\end{remark}

The following integration by parts formula is contained in \cite{DRV, War-ACE}. Let $\varphi\in H^{s}(\RR^N)$ be such that $(-\Delta)^s \varphi\in L^2(\Omega)$ and $\mathcal N_s\varphi\in L^2(\Omb)$. Then for every $\psi\in H^{s}(\RR^N)$, the  identity
\begin{align}\label{Int-Part}
\frac{C_{N,s}}{2}\int\int_{\RR^{2N}\setminus(\Omc)^2}&
\frac{(\varphi(x)-\varphi(y))(\psi(x)-\psi(y))}{|x-y|^{N+2s}}\;\dx\, \dy\notag\\
=&\int_{\Omega}\psi(-\Delta)^s\varphi\;\dx+\int_{\Omc}\psi\mathcal N_s \varphi\;\dx
\end{align}
holds.

Observing that
\begin{align*}
\RR^{2N}\setminus(\RR^N\setminus\Omega)^2=(\Omega\times\Omega)\cup (\Omega\times(\RR^N\setminus\Omega))\cup((\RR^N\setminus\Omega)\times\Omega),
\end{align*}
we have that if $\varphi=0$ in $\Omc$ or $\psi=0$ in $\Omc$, then
\begin{align*}
&\int\int_{\RR^{2N}\setminus(\RR^N\setminus\Omega)^2}
\frac{(\varphi(x)-\varphi(y))(\psi(x)-\psi(y))}{|x-y|^{N+2s}}\dx\,\dy\\
=&\int_{\R^N}\int_{\R^N}\frac{(\varphi(x)-\varphi(y))(\psi(x)-\psi(y))}{|x-y|^{N+2s}}\dx\,\dy.
\end{align*}

Throughout the remainder of the paper,   we shall let the bilinear form $\mathcal F:H_0^{s}(\Om)\times H_0^{s}(\Om)\to\RR$ be given by
\begin{equation}\label{defF}
\mathcal{F}(\varphi,\psi)\coloneqq \frac{C_{N,s}}{2}\dis \int_{\RR^N}\int_{\RR^N}\frac{(\varphi(x)-\varphi(y))(\psi(x)-\psi(y))}{|x-y|^{N+2s}}\;\dx\, \dy.
\end{equation}

Next, we introduce the classical first order Sobolev space
\begin{align*}
H^1(\Omega):=\left\{u\in L^2(\Om):\;\int_{\Om}|\nabla u|^2\;\dx<\infty\right\}
\end{align*}
which is endowed with the norm defined by
\begin{align*}
\|u\|_{H^1(\Omega)}=\left(\int_{\Omega}|u|^2\;\dx+\int_{\Om}|\nabla u|^2\;\dx\right)^{\frac 12}.
\end{align*}

In order to study the  solvability of~\eqref{Eq:main_A1}, we  shall also need  the following function space
\begin{equation} \label{eq.defXOmega}
H_0^1(\Omega) \coloneqq \Big\{w\in H^1(\R^N):\,\text{$w \equiv 0$ in $\Omc$}\Big\},
\end{equation}
which is a  (real) Hilbert space endowed with the scalar product
\[
 \int_{\Omega} \nabla w\cdot\nabla \varphi\,\dx,
\]
and  associated norm
\begin{align}\label{H1}
\|\varphi\|_{H_0^1(\Omega)}\coloneqq \|\nabla \varphi\|_{L^2(\Omega)}.
\end{align}

Furthermore, the classical Poincar\'{e} inequality holds in $H_0^1(\Omega)$. That is,  there is a constant $C > 0$ such that
\begin{equation} \label{eq.PoincareX}
\|\varphi\|_{L^2(\Omega)} \leq C\,\|\varphi\|_{H_0^1(\Omega)}\qquad
\text{for all $\varphi\in H_0^1(\Omega)$}.
\end{equation}
We shall denote by $H^{-1}(\Omega)$ the dual space of $H_0^1(\Omega)$ with respect to the pivot space $L^2(\Omega)$ so that we have the following continuous and dense embeddings:
\begin{equation*}
 H_0^{1}(\Om)\hookrightarrow L^2(\Omega)\hookrightarrow H^{-1}(\Om).
 \end{equation*}
Here also, if $\Omega$ is bounded and has a Lipschitz continuous boundary,  then by \cite[Chapter 1]{Gris}
\begin{align*}
 H_0^1(\Omega)=\overline{\mathcal D(\Omega)}^{H^1(\Omega)}.
 \end{align*}
In addition, under the same assumption on $\Omega$, every function $u\in H^1(\Omega)$ has a trace $u|_{\pOm}$ that belongs to $H^{1/2}(\pOm)$,  and the mapping trace
\begin{align}\label{trace}
H^1(\Omega)\to H^{\frac 12}(\pOm),\;\; u\mapsto u|_{\pOm}
\end{align}
is continuous and surjective.\medskip

Now assume that $\Omega$ is a bounded open set with a Lipschitz continuous boundary. Then every function $u\in H^2(\Omega)$ has a normal derivative $\partial_\nu u:=\nabla u\cdot\nu$ that belongs to $H^{1/2}(\pOm)\hookrightarrow L^2(\pOm,\sigma)$. Here, $\sigma$ denotes the restriction to $\pOm$ of the $(N-1)$-dimensional Hausdorff measure which coincides with the Lebesgue surface measure since $\Omega$ has a Lipschitz continuous boundary, and $\nu$ is the outer normal vector at the boundary. More precisely, there is a constant $C>0$ such that for every $u\in H^2(\Omega)$,
\begin{align}\label{ND}
\|\partial_\nu u\|_{L^2(\pOm)}\le C\|u\|_{H^2(\Omega)}.
\end{align}
In addition, for every $u\in H^2(\Omega)$ and $v\in H^1(\Omega)$, the integration by parts formula
\begin{align}\label{IBP-L}
-\int_{\Omega}v\Delta u \;\dx=\int_{\Omega}\nabla u\cdot\nabla v\;\dx-\int_{\pOm}v\partial_\nu u\; \mathrm{d}\sigma
\end{align}
holds.  We refer to \cite[Lemma 1.5.3.7]{Gris} for the proof. 

Assume that $\Omega$ is a bounded open set of class $C^{1,1}$ and let $E(\Delta,L^2(\Omega)):=\{u\in H^1(\Omega):\;\Delta u\in L^2(\Omega)\}$. Then $C^\infty(\bOm)$ is dense in $E(\Delta,L^2(\Omega))$ and every $u\in E(\Delta,L^2(\Omega))$  has a normal derivative $\partial_\nu u:=\nabla u\cdot\nu$ that belongs to $H^{-1/2}(\pOm):=(H^{1/2}(\pOm,\sigma))^\star$.  More precisely, there is a constant $C>0$ such that for every $u\in E(\Delta,L^2(\Omega))$,
\begin{align}\label{ND-2}
\|\partial_\nu u\|_{H^{-1/2}(\pOm)}\le C\left(\|u\|_{L^2(\Omega)}+\|\Delta u\|_{L^2(\Omega)}\right).
\end{align}
In addition, for every $u\in E(\Delta,L^2(\Omega))$ and $v\in H^1(\Omega)$, the integration by parts formula
\begin{align}\label{IBP-L-2}
-\int_{\Omega}v\Delta u \;\dx=\int_{\Omega}\nabla u\cdot\nabla v\;\dx-\langle \partial_\nu u,v\rangle_{H^{-1/2}(\pOm),H^{1/2}(\pOm)}
\end{align}
holds (see e.g.  \cite[Formula 1.5.3.30, pp 62]{Gris}).   We refer to \cite{Adams,Bur,GiTr,Gris,MaPo} and the references therein for more details on this topic.\medskip

Throughout the remainder of the paper, without any mention, we shall assume that $\Omega\subset\RR^N$ is a bounded domain with a  smooth  boundary $\pOm$. Under this assumption, we have the following continuous and dense embedding for every $0<s<1$ (see e.g. \cite{Gris,NPV}):
\begin{align}\label{H1-to-Hs}
H_0^1(\Omega)\hookrightarrow H_0^s(\Omega).
\end{align}
In view of \eqref{H1} and \eqref{H1-to-Hs}, we can deduce that
\begin{equation}\label{scalaireX}
(\varphi,\psi)_{H_0^1(\Om)}:=\mathcal{F}(\varphi,\psi)+
\int_\Omega \nabla \varphi\cdot\nabla \psi \, \dx
\end{equation}
defines a scalar product on $H_0^1(\Omega)$ with associated norm
\begin{align}\label{H1-1}
\|\varphi\|_{H_0^1(\Omega)}\coloneqq \left(\mathcal{F}(\varphi,\varphi)+ \int_\Omega |\nabla \varphi|^2\, \dx\right)^{\frac 12}.
\end{align}
The norm given in \eqref{H1-1} is equivalent to the one given in \eqref{H1}.

Throughout the remainder of the article, otherwise stated, $H_0^1(\Omega)$ will be endowed with the scalar product and norm given in \eqref{scalaireX} and \eqref{H1-1}, respectively.

We conclude this section by giving the following result. Let $T>0$ be a real number and
set
\begin{equation}\label{defW0T}
W(0,T):= \left\{\zeta \in L^2((0,T);H_0^1(\Omega)): \zeta_{t} \in L^2((0,T);H^{-1}(\Omega))\right\}.
\end{equation}
Since $H_0^1(\Omega)$ is a real Hilbert space, it follows from Lions--Magenes  \cite[Theorem II.5.12]{BF13} and \cite{Tommaso} that
 $W(0,T)$ endowed with the  norm
\begin{equation}\label{normW0T}
\|\zeta\|_{W(0,T)}\coloneqq \left(\|\zeta\|^2_{L^2((0,T);H_0^1(\Omega))}+\|\zeta_t\|^2_{L^2((0,T);H^{-1}(\Omega))}\right)^{\frac 12},
\end{equation}
is a Hilbert space. Moreover, we have the following continuous embedding:
\begin{equation}\label{contWTA}
W(0,T)\hookrightarrow C([0,T],L^2(\Omega)).
\end{equation}
Furthermore, from \eqref{eq.PoincareX} and the Lions--Aubin Lemma \cite[Theorem II.5.16]{BF13}
 the following embedding is compact 
  \begin{equation}\label{Lions-compactness}
W(0,T) \hookrightarrow L^2((0,T); L^2(\Omega)).
  	\end{equation}

\section{Well-posedness of the state equation} \label{Sec:Wellpo}

In this section we are interested in establishing some existence, uniqueness and regularity results of the state equation  \eqref{Eq:main_A1}
that will be needed in the proof of the existence of minimizers to the optimal control problem \eqref{Eq:main_A}. We recall that without any mention, $\Omega\subset\RR^N$ ($N\ge 1$) is a bounded domain with a smooth  boundary $\pOm$.

\subsection{The elliptic problem}\label{sec-31}
We start with the stationary problem. That is,  we consider the following non-homogeneous Dirichlet problem associated with the operator $\mathscr{L}$, as defined in~\eqref{LOPERATO1}. That is,
\begin{equation}\label{EDP}
\begin{cases}
 \mathscr{L}w =f\;\;&\mbox{ in }\;\Omega,\\
w=g_1&\mbox{ on }\;\pOm,\\
w=g_2&\mbox{ in }\;\Omc.
\end{cases}
\end{equation}
Even if we are not considering control problems for elliptic equations, the complete analysis of \eqref{EDP} will be crucial in the study of the associated time dependent problem, and is also interesting in its own, independently of the applications given in the present paper.

To introduce our notion of solutions to the system \eqref{EDP}, we start with the simple case $g_1=0$ on $\pOm$ and $g_2=0$ in $\Omc$.

\begin{definition} \label{def.weaksol}
Let $f\in H^{-1}(\Omega)$, $g_2=0$ in $\Omc$ and $g_1=0$ on $\pOm$.  A function $w\in H_0^1(\Omega)$ is said to be a \emph{weak solution of}  \eqref{EDP}  if for every  function $\varphi\in H_0^1(\Omega)$, the identity
 \begin{equation} \label{eq.weaksoldef}
\int_{\Omega} \nabla w\cdot\nabla \varphi\,\dx +\mathcal F(w,\varphi)
 = \langle f,\varphi\rangle_{H^{-1}(\Omega), H_0^1(\Omega)}
\end{equation}
holds, where we recall that the bilinear form $\mathcal F$ has been defined in \eqref{defF}.
\end{definition}
The following existence result can be  established by using the classical Lax-Milgram Lemma. We refer to \cite[Theorem 1.1]{BDVV2020} for more details.

\begin{proposition}\label{proposi-33-1}
Let $g_2=0$ in $\Omc$ and $g_1=0$ on $\pOm$.  Then for every $f\in H^{-1}(\Omega)$, there is a unique function $w\in H_0^{1}(\Omega)$ satisfying \eqref{EDP} in the sense of Definition~\ref{def.weaksol}. In addition, there is a constant $C>0$ such that
\begin{align*}
\|w\|_{H_0^{1}(\Omega)}\le C\|f\|_{H^{-1}(\Omega)}.
\end{align*}
\end{proposition}

Next, we consider the case of smooth non-zero boundary-exterior  data.

\begin{definition}\label{def-sol}
Let $f\in H^{-1}(\Omega)$. Let $g_2\in H^{1}(\Omb)$ and  $g_1\in H^{1/2}(\pOm)$ be such that $g_1=g_2|_{\pOm}$. Let $G\in H^{1}(\RR^N)$ be such that $G|_{\Omb}=g_2$.
A function $w\in H^{1}(\RR^N)$ is said to be a weak solution of  \eqref{EDP} if $w-G\in H_0^{1}(\Om)$ and the identity
\begin{align}\label{Weakforulinho}
\int_{\Omega} \nabla w\cdot\nabla \varphi\,\dx +\mathcal F(w,\varphi)= \langle f,\varphi\rangle_{H^{-1}(\Omega), H_0^1(\Omega)}
\end{align}
holds for every $\varphi\in H_0^{1}(\Omega)$.
\end{definition}
The following existence result can be easily established. We also refer to \cite[Corollary 2.6]{BDVV2020} where the same result has been proved under the assumption that the function $g_2$ is smooth enough. The version given here can be obtained by using  \cite[Corollary 2.6]{BDVV2020} and an approximation argument.

\begin{proposition}\label{proposi-33}
Let and $g_1$ and $g_2$ be as in Definition \ref{def-sol} and $f\in H^{-1}(\Omega)$. Then, there is a unique $w\in H^{1}(\RR^N)$ satisfying \eqref{EDP} in the sense of Definition~\ref{def-sol}. In addition, there is a constant $C>0$ such that
\begin{align*}
\|w\|_{H^{1}(\RR^N)}\le C\left(\|f\|_{H^{-1}(\Omega)}+\|g_2\|_{H^{1}(\Omb)}\right).
\end{align*}
\end{proposition}

\begin{remark}
We notice that in the situation of Definitions \ref{def.weaksol}- \ref{def-sol}, and Propositions \ref{proposi-33-1}-\ref{proposi-33}, the system \eqref{EDP} becomes
\begin{equation*}
\mathscr{L}w =f\;\;\mbox{ in }\;\Omega,\;\;\; w=g_2\mbox{ in }\;\Omc,
\end{equation*}
that is, the condition on $\pOm$ is not needed in order to have a well-posed problem.
\end{remark}

Finally, we consider the case of non-smooth boundary-exterior  data.  This case has not been discussed in  \cite{BDVV2020} and we need a new notion of solutions that we introduce next.

\begin{definition}\label{EVWS}
Let $f\in H^{-1}(\Omega)$, $g_1\in L^2(\pOm)$, and $g_2\in L^2(\Omc)$. A function $w\in L^2(\RR^N)$ is called a very-weak solution (or a solution by transposition) of  \eqref{EDP},  if the identity
\begin{align}\label{e34}
\int_{\Omega}w \mathscr{L}\varphi\;\dx= \langle f,\varphi\rangle_{H^{-1}(\Omega), H_0^1(\Omega)} -\int_{\pOm}g_1\partial_\nu\varphi\;\mathrm{d}\sigma-\int_{\Omb}g_2\mathcal N_s\varphi\;\dx
\end{align}
holds, for every $\varphi\in \mathbb V:=\Big\{\varphi\in H_0^1(\Omega):\; \mathscr{L}\varphi\in L^2(\Omega)\Big\}$.
\end{definition}

We notice that Definition \ref{EVWS} of very-weak solutions makes sense if every function $\varphi\in\mathbb V$ satisfies $\partial_\nu\varphi \in L^2(\pOm)$,  and $\mathcal N_s\varphi\in L^2(\Omb)$.

We have the following existence theorem which is the main result of this section.

\begin{theorem}\label{th-36}
Let $0<s\le 3/4$.  Then for every $f\in H^{-1}(\Omega)$, $g_1\in L^2(\pOm)$ and $g_2\in L^2(\Omc)$, the system \eqref{EDP} has a unique very-weak solution $w\in L^2(\RR^N)$ in the sense of Definition \ref{EVWS}, and  there is a constant $C>0$ such that
\begin{align}\label{e35}
\|w\|_{L^2(\RR^N)}\le C\left(\|f\|_{H^{-1}(\Omega)}+\|g_1\|_{L^2(\pOm)}+\|g_2\|_{L^2(\Omc)}\right).
\end{align}
In addition, if $g_1$ and $g_2$ are as in Definition \ref{def-sol}, then the following assertions hold.
\begin{enumerate}
\item Every weak solution of \eqref{EDP}  is also a very-weak solution.
\item Every very-weak solution of \eqref{EDP}  that belongs to $H^1(\RR^N)$ is also a weak solution.
\end{enumerate}
\end{theorem}
\begin{proof}
The proof follows similarly as the case of the single fractional Laplace operator (see e.g. \cite{AKW}) with the exception of the restriction on $s$ which is an important and delicate step. We include the full proof for the sake of completeness. We proceed in several steps.\medskip

{\bf Step 1}:  Let $\mathbb A$ be the realization of $ \mathscr{L} $ in $L^2(\Omega)$ with zero Dirichlet exterior condition, that is,
\begin{equation}\label{Rea}
D(\mathbb A)=\mathbb V\coloneqq \{u\in H_0^1(\Omega):\; ( \mathscr{L}u)|_{\Omega}\in L^2(\Omega)\},\;\;
\mathbb Au=( \mathscr{L}u)|_{\Om}\;\mbox{ in }\,\Omega.
\end{equation}
We shall give a detailed description of this operator later.
 Notice that $\|v\|_{\mathbb V} := \|\mathscr{L}v\|_{L^2(\Om)}$ defines an equivalent norm on  $\mathbb V$.  This follows from the fact that the operator $\mathbb A$ is invertible, has a compact resolvent, and its first eigenvalue is strictly positive (see the proof of Theorem \ref{Thm1} below for more details).  We claim that
 \begin{align*}
 \mathbb V=\Big\{\varphi\in H_0^1(\Omega):\; \mathscr{L}\varphi\in L^2(\Omega),\;\; \partial_\nu \varphi\in L^2(\partial\Omega)\; \mbox{and }\;\mathcal  N_s\varphi\in L^2(\Omb)\Big\}.
 \end{align*}
 It suffices to show that $\partial_\nu \varphi\in L^2(\partial\Omega)$  and $\mathcal N_s\varphi\in L^2(\Omb)$ for every $\varphi\in\mathbb V$. Indeed, let $\varphi\in \mathbb V$. We have two cases.
\begin{itemize}
 \item Case $0<s<1/2$. Since $\varphi\in H_0^1(\Omega)\hookrightarrow H_0^s(\Om)$, it follows from the regularity result contained in \cite[Theorem 4.1]{G-JFA} that $(-\Delta)^s\varphi\in L^2(\Omega)$. Thus, $\mathcal N_s\varphi\in L^2(\Omb)$ by Remark \ref{rem-21}.  As $\mathscr{L}u\in L^2(\Omega)$, this also implies that $\Delta\varphi\in  L^2(\Omega)$. Since $\Omega$ is assumed to be smooth, using the well-known elliptic regularity results for the Laplace operator, (see e.g. \cite{GiTr}), we have that $\varphi\in  H^{2}(\Omega)$. Thus, $\partial_\nu\varphi\in  H^{1/2}(\partial\Omega)\hookrightarrow L^2(\partial\Omega)$.\medskip

\item Case $1/2\le s<1$.  Since $\varphi\in H_0^1(\Omega)$, it follows from  \cite[Theorem 4.1]{G-JFA} again that $(-\Delta)^s\varphi\in H^{1-2s}(\Omega)$. Hence,  $\Delta\varphi\in  H^{1-2s}(\Omega)$ and this implies that $\varphi\in  H^{3-2s}(\Omega)$. Thus, $\partial_\nu\varphi\in  H^{3/2-2s}(\partial\Omega)\hookrightarrow L^2(\partial\Omega)$ if $3/2-2s\ge 0$. Since $\varphi\in  H^{3-2s}(\Omega)$, it follows from \cite{G-JFA} again that if $3-4s\ge 0$, then $(-\Delta)^s\varphi\in L^2(\Om)$.  Since $\mathscr{L}\varphi\in L^2(\Omega)$, we can deduce that $\Delta\varphi\in L^2(\Omega)$. From elliptic regularity for the Laplace operator operator we can conclude that $\varphi\in H^2(\Omega)$. Thus, $\partial_\nu \varphi\in H^{1/2}(\partial\Omega)\hookrightarrow L^2(\partial\Omega)$.
\end{itemize}
Since we have assumed that $0<s\le 3/4$, it follows that $3/2-2s\ge 0$. and $3-4s\ge 0$.  Thus, using Remark \ref{rem-21}, 
we can conclude that $\partial_\nu \varphi\in L^2(\partial\Omega)$  and $\mathcal N_s\varphi\in L^2(\Omb)$. The claim is proved.\medskip

{\bf Step 2}:  We apply the  Babu\v{s}ka-Lax-Milgram theorem.
 Let $\mathbb F:L^2(\Omega)\times \mathbb V\to\RR$ be the bilinear form defined  by
  \begin{align*}
  \mathbb F(u,v)\coloneqq \int_{\Om}  u \mathscr{L} v\;\dx.
  \end{align*}
 It is clear that $\mathbb F$ is  bounded
  on $L^2(\Om) \times \mathbb V$. That is, 
  \begin{align*}
 \left|\mathbb F(u,v)\right|\le \|u\|_{L^2(\Omega)}\| \mathscr{L}  v\|_{L^2(\Omega}=\|u\|_{L^2(\Omega)}\|v\|_{\mathbb V},\;\;\forall\; (u,v)\in L^2(\Omega)\times \mathbb V.
  \end{align*}

Next, we show the inf-sup condition.  
 Letting
\[
u \coloneqq \frac{ \mathscr{L} v}{\| \mathscr{L} v\|_{L^2(\Om)}} \in L^2(\Om),
\]
we obtain that
  \begin{align*}
   \sup_{u \in L^2(\Om) , \|u\|_{L^2(\Om)}=1} |(u, \mathscr{L}  v)_{L^2(\Om)}|
    \ge \frac{|(\mathscr{L} v, \mathscr{L} v)_{L^2(\Om)}|}{\| \mathscr{L} v\|_{L^2(\Om)}} =\| \mathscr{L} v\|_{L^2(\Om)} = \|v\|_{\mathbb V} .
  \end{align*}
Next, we choose $v \in \mathbb V$ as the unique weak solution of the Dirichlet problem
\begin{align*}
 \mathscr{L} v = \frac{h}{\|h\|_{L^2(\Om)}}\;\;\mbox{ in }\;\Omega\;\mbox{ for some }\;0\ne h \in L^2(\Om).
  \end{align*}
  Then we readily obtain that
  \[
   \sup_{v\in\mathbb V, \|v\|_{\mathbb V}=1} |(h, \mathscr{L} v)_{L^2(\Om)}|
    \ge \frac{|(h,h)_{L^2(\Om)}|}{\|h\|_{L^2(\Om)}} = \|h\|_{L^2(\Om)} >0\;\mbox{  for all }\;0\ne h\in L^2(\Omega).
  \]

Finally, we have to show that the right-hand-side in \eqref{e34}
defines a linear continuous functional on $\mathbb V$. Indeed, applying the H\"older inequality, the fact that  $\partial_\nu$ maps $\mathbb V$ continuously into $L^2(\pOm)$, and $\mathcal N_s$ maps $\mathbb V$ continuously into $L^2(\Omb)$ (by Step 1), we obtain that there is a constant $C>0$ such that
  \begin{align}\label{NOR-EST}
 &  \left|\langle f,v\rangle_{H^{-1}(\Om),H_0^1(\Om)} -\int_{\pOm}g_1\partial_\nu v\;\mathrm{d}\sigma-\int_{\RR^N\setminus\Om} g_2 \mathcal{N}_s v\;\dx \right|\notag\\
   \le&  \|f\|_{H^{-1}(\overline\Om)} \|v\|_{H^{1}_0(\Om)}+\|g_1\|_{L^2(\pOm)}\|\partial_\nu v\|_{L^2(\pOm)}+ \|g_2\|_{L^2(\Omb)} \|\mathcal{N}_s v\|_{L^2(\Omb)}\notag \\
   \le &C\left(  \|f\|_{H^{-1}(\overline\Om)} + \|g_1\|_{L^2(\pOm)}+\|g_2\|_{L^2(\Omb)}\right)\|v\|_{\mathbb V}.
  \end{align}
  In view of \eqref{NOR-EST}, we can deduce that the right-hand-side in \eqref{e34}
  defines a linear continuous functional on the Hilbert space $\mathbb V$.
Therefore, all the requirements of the Babu\v{s}ka--Lax-Milgram
theorem hold. Thus, we can conclude that there exists a unique $w \in L^2(\Om)$ satisfying \eqref{e34}. Letting $w\coloneqq g_2$ in $\Omb$, we have that $w\in L^2(\RR^N)$ and satisfies \eqref{e34}. We have shown the existence and uniqueness of a very-weak solution to the system \eqref{EDP}. \medskip

{\bf Step 3}: Next, we show the estimate \eqref{e35}. Let $w\in L^2(\RR^N)$ be a very-weak solution of \eqref{EDP}. Let $v\in \mathbb V$ and set $w\coloneqq \mathbb Av= \mathscr{L}  v$ in $\Omega$. Taking $v$ as a test function in \eqref{e34}, using \eqref{NOR-EST}, \eqref{NND}, \eqref{H1-to-Hs}, and \eqref{ND},
we can deduce that there is a constant $C>0$ such that
 \begin{align*}
 \|w\|_{L^2(\Omega)}^2
\le &C\left( \|f\|_{H^{-1}(\Om)}+\|g_1\|_{L^2(\pOm)}+\|g_2\|_{L^2(\Omb)}\right)\|v\|_{\mathbb V}\\
= &C\left( \|f\|_{H^{-1}(\Om)}+\|g_1\|_{L^2(\pOm)}+\|g_2\|_{L^2(\Omb)}\right)\|w\|_{L^2(\Om)}.
 \end{align*}
Thus,
\begin{align*}
\|w\|_{L^2(\Omega)}\le C\left( \|f\|_{H^{-1}(\Om)}+\|g_1\|_{L^2(\pOm)}+\|g_2\|_{L^2(\Omb)}\right).
\end{align*}
Since $w=g_2$ in $\Omb$, it follows from the preceding estimate that
\begin{align*}
\|w\|_{L^2(\RR^N)}\le C\left( \|f\|_{H^{-1}(\Om)}+\|g_1\|_{L^2(\pOm)}+\|g_2\|_{L^2(\Omb)}\right).
\end{align*}
We have shown the estimate \eqref{e35} and this completes the proof of the first part. \medskip

{\bf Step 4}: We prove the assertions (a) and (b) of the theorem. Assume that $g\coloneqq g_2\in H^1(\Omb)$, $g_1\in H^{1/2}(\pOm)$, and that $g_2|_{\pOm}=g_1$.\medskip

(a) Let $w\in H^{1}(\RR^N)\hookrightarrow L^2(\RR^N)$ be a weak solution of \eqref{EDP}. It follows from the definition that $w=g_2$ in $\Omb$, $w|_{\pOm}=g_2|_{\pOm}=g_1$ on $\pOm$,  and
\begin{align}\label{eEE1}
\int_{\Om}\nabla w\cdot\nabla v\;\dx+\mathcal F(w,v)=\langle f,v\rangle_{H^{-1}(\Om),H_0^1(\Om)},
\end{align}
for every $v\in \mathbb V$.
Since $v=0$ in $\Omb$, we have that
\begin{align}\label{eEE2}
\int_{\RR^N}\int_{\RR^N}&\frac{(w(x)-w(y))(v(x)-v(y))}{|x-y|^{N+2s}}\;\dx \;\dy\notag\\
&=\int\int_{\RR^{2N}\setminus(\RR^N\setminus\Omega)^2}\frac{(w(x)-w(y))(v(x)-v(y))}{|x-y|^{N+2s}}\;\dx\;\dy.
\end{align}

Let $v\in H^2(\Omega)\cap H_0^1(\Omega)$. Notice that both $\Delta v$ and $(-\Delta)^sv$ belong to $L^2(\Omega)$, $\partial_\nu v\in L^2(\partial\Omega)$, and $\mathcal N_sv\in L^2(\Omb)$. Therefore,
using \eqref{eEE1}, \eqref{eEE2}, the integration by parts formulas \eqref{Int-Part}-\eqref{IBP-L},  we get that for every $v\in H^2(\Omega)\cap H_0^1(\Omega)$,
\begin{align}\label{WW}
&\langle f,v\rangle_{H^{-1}(\Om),H_0^1(\Om)}\notag\\
=&\int_{\Omega}\nabla w\cdot\nabla v\;\dx+\frac{C_{N,s}}{2}\int_{\RR^N}\int_{\RR^N}\frac{(w(x)-w(y))(v(x)-v(y))}{|x-y|^{N+2s}}\;\dx\; \dy\notag\\
=&\int_{\Om}w(-\Delta)v+\int_{\pOm}w\partial_\nu v\;\mathrm{d}\sigma+\int_{\Omega}w(-\Delta)^sv\;\dx+\int_{\Omb}w\mathcal N_sv\;\dx\notag\\
=&\int_{\Om}w \mathscr{L} v+\int_{\pOm}g_1\partial_\nu v\;\mathrm{d}\sigma+\int_{\Omb}g_2\mathcal N_s v\;\dx.
\end{align}
Since $H^2(\Omega)\cap H_0^1(\Omega)$ is dense in $\mathbb V$, we have that \eqref{WW} remains true for every $v\in\mathbb V$. Thus, $w$ is a very-weak solution of \eqref{EDP}.\medskip

(b) Finally, let $w$ be a very-weak solution of \eqref{EDP} and assume that $w\in H^{1}(\RR^N)$. Since $w=g_2$ in $\Omb$, we have that $g_2\in H^{1}(\Omb)$. Let $\widetilde g\in H^{1}(\RR^N)$ be such that $\widetilde g|_{\Omb}=g_2$. Then clearly $(w-\widetilde g)\in H_0^{1}(\Om)$. Since $w$ is a very-weak solution of \eqref{EDP}, it follows from the definition that for every $v\in \mathbb V$,
\begin{align}\label{EEe3}
\int_{\Omega}w \mathscr{L} v\;\dx=\langle f,v\rangle_{H^{-1}(\Om),H_0^1(\Om)} -\int_{\pOm}g_1\partial_\nu v\;\mathrm{d}\sigma-\int_{\Omb}g_2\mathcal N_sv\;\dx.
\end{align}
In particular, \eqref{EEe3} holds for every $v\in H^2(\Omega)\cap H_0^1(\Omega)$.
Let then $v\in H^2(\Omega)\cap H_0^1(\Omega)$.
 Since $v\in H^{1}(\RR^N)$ and $v=0$ in $\Omb$, it follows from \eqref{Int-Part} and \eqref{eEE2} that
\begin{align}\label{EEe4}
&\frac{C_{N,s}}{2} \int_{\RR^N}\int_{\RR^N}\frac{(w(x)-w(y))(v(x)-v(y))}{|x-y|^{N+2s}}\;\dx\; \dy\\
&=\int_{\Omega}w(-\Delta)^sv\;\dx+\int_{\Omb}w\mathcal N_sv\;\dx
=\int_{\Omega}w(-\Delta)^sv\;\dx+\int_{\Omb}g_2\mathcal N_sv\;\dx.\notag
\end{align}
It also follows from \eqref{IBP-L} that
\begin{align}\label{EEe5}
\int_{\Omega}\nabla w\cdot\nabla v\;\dx=&-\int_{\Omega}w\Delta v\;\dx+\int_{\pOm}w\partial_\nu v\;\mathrm{d}\sigma\notag\\
=&-\int_{\Omega}w\Delta v\;\dx+\int_{\pOm}g_1\partial_\nu v\;\mathrm{d}\sigma.
\end{align}
Combining \eqref{EEe3}-\eqref{EEe4} and \eqref{EEe5}, we get that for every $v\in H^2(\Omega)\cap H_0^1(\Omega)$,
\begin{align}\label{eEE5}
\int_{\Omega}\nabla w\cdot\nabla v\;\dx+\mathcal F(w,v)=\langle f,v\rangle_{H^{-1}(\Om),H_0^1(\Om)}.
\end{align}
Since $H^2(\Omega)\cap H_0^1(\Omega)$ is dense in $H_0^{1}(\Om)$, we have that \eqref{eEE5} remains true for every $v\in H_0^{1}(\Om)$. We have shown that $w$ is a weak solution of \eqref{EDP} and the proof is finished.
\end{proof}

We conclude this section with the following remark.

\begin{remark}\label{rem-37}
We observe the following facts.
\begin{enumerate}
\item We notice that in Definition \ref{EVWS} of very-weak solutions, we do not require that the function $w$ has a well-defined trace on $\partial\Omega$ and that $w|_{\partial\Omega}=g_1$, for that reason the regularity of $w$ cannot be improved.

\item But if $w$ has a well-defined trace on $\partial\Omega$ and $w|_{\partial\Omega}=g_1\in L^2(\pOm)$, then the regularity of $w$ can be improved. Indeed, using well-known trace theorems (see e.g. \cite{GM-2011}) we can deduce that $w\in L^2(\R^N)\cap H^{1/2}(\Omega)$.

\item Let $\mathbb V$ be the space defined in \eqref{Rea}. If $0<s\le 3/4$, then $\mathbb V\subset H^2(\Omega)\cap H_0^1(\Omega)$.  Indeed,  let $\varphi\in\mathbb V$. If $0<s<1/2$, it follows from the proof of Theorem \ref{th-36} Step 1 that $\varphi\in H^2(\Omega)\cap H_0^1(\Omega)$. If $1/2\le s\le 3/4$, then the proof of Theorem \ref{th-36} Step 1 shows again that $\varphi\in H^{3-2s}\cap H_0^1(\Omega)$. Using \cite{G-JFA}, we get that, in fact $(-\Delta)^s\varphi\in L^2(\Omega)$. This  implies that $\Delta\varphi\in L^2(\Omega)$. Thus, $\varphi\in H^2(\Omega)\cap H_0^1(\Omega)$ by using elliptic regularity results for the Laplace operator.

\item Consider the following Dirichlet problem: Find $u\in H_0^1(\Omega)$ satisfying
\begin{equation*}
\mathscr{L}u =f\;\mbox{ in } \Omega.
\end{equation*}
Due to the presence of the fractional Laplace operator $(-\Delta)^s$, even if $f$ is smooth, classical bootstrap argument cannot be used to improve the regularity of the solution $u$.  This follows from the fact that even if $f$ is smooth enough,  if $1/2\le s<1$,  then a function $v\in H_0^s(\Omega)$ satisfying $(-\Delta)^sv=f$ in $\Omega$ only belongs to $\cap_{\varepsilon>0}H^{2s-\varepsilon}(\Omega)$ and does not belong to $H^{2s}(\Omega)$.
We refer to the papers \cite{BoNo,Ro-Sj} for more details on this topic.
This suggests that for functions in the space $\mathbb V$ given in \eqref{Rea}, the regularity discussed in Step 1 in the proof of Theorem \ref{th-36} cannot be improved. At least, we do not know how to improve the regularity of functions belonging to $\mathbb V$.

\item In the case $3/4<s<1$, if the function $g_1$ is smooth, says, $g_1\in H^{2s-3/2}(\pOm)$, then we may replace  \eqref{e34} in the definition of very-weak solutions by the expression:
\begin{align*}
\int_{\Omega}w \mathscr{L}\varphi\;\dx=& \langle f,\varphi\rangle_{H^{-1}(\Omega), H_0^1(\Omega)} -\langle g_1,\partial_\nu\varphi\rangle_{H^{2s-3/2}(\pOm),H^{3/2-2s}(\pOm)}\\
&-\int_{\Omb}g_2\mathcal N_s\varphi\;\dx
\end{align*}
holds, for every $\varphi\in\mathbb V$.
In that case,  Theorem \ref{th-36} will be valid for every $0<s<1$.  But recall that the main objective of the paper is to study the minimization problem \eqref{Eq:func_A1} and our control function $u_1$ does not enjoy such a regularity.
\end{enumerate}
\end{remark}

\subsection{The parabolic problem}\label{sec-32}

First, we consider the following auxiliary problem:
\begin{equation}\label{p1}
\begin{cases}
\phi_{t} + \mathscr{L}\phi =  f & \mbox{ in }\; Q,\\
\phi = 0 &\mbox{ in }\; \Sigma, \\
\phi(\cdot,0) = \phi_{0}, &\mbox{ in }\; \Omega.
\end{cases}
\end{equation}
We shall denote by $\langle\cdot,\cdot\rangle$ the duality pairing between $H^{-1}(\Om)$ and $H_0^1(\Om)$. Here is our notion of weak solutions to the system \eqref{p1}.

\begin{definition}\label{weaksolution}
Let $f\in L^2(Q)$ and $\phi_{0}\in L^2(\Omega)$. We shall say that a function
$\phi\in \mathbb U\coloneqq L^2((0,T);H_0^1(\Om))\cap H^1((0,T);H^{-1}(\Omega))$ is a weak solution to \eqref{p1}, if  $\phi(\cdot,0)=\phi_0$ a.e. in $\Omega$ and
 the  equality
\begin{align}\label{Eq-Def31}
\langle \phi_t,\zeta\rangle_{H^{-1}(\Om),H_0^1(\Om)} +\int_{\Om}\nabla \phi\cdot\nabla \zeta\, \dx
 + \displaystyle \mathcal{F}(\phi,\zeta)= \int_{\Om} f\zeta\;\dx
\end{align}
holds, for every $\zeta \in H_0^1(\Om)$ and almost every $t\in (0,T)$.
\end{definition}

\begin{remark}
It is worthwhile noticing that if $\phi\in \mathbb U$ is a weak solution of \eqref{p1} with $f\in L^2(Q)$, then $\phi\in W(0,T)$. Thus, by \eqref{contWTA} $\phi\in C([0,T];L^2(\Omega))$, so that $\phi(\cdot,0)=\phi_0$ a.e. in $\Omega$ makes sense.
\end{remark}

Throughout the remainder of the article, we shall let
\begin{align*}
\mathbb U\coloneqq L^2((0,T);H_0^1(\Om))\cap H^1((0,T);H^{-1}(\Omega)).
\end{align*}
Now, we are in position to state the well-posedness of \eqref{p1}.

\begin{theorem}\label{Thm1}
Let $f\in L^2(Q)$ and $\phi_{0} \in L^{2}(\Omega)$. Then, there exists a unique weak solution $\phi\in \mathbb U$ to \eqref{p1} in the sense of Definition~\ref{weaksolution}.
In addition, there is a constant $C>0$ such that
\begin{equation}\label{estimation1}
\|\phi\|_{\mathbb U}\leq C\left(\|\phi_{0}\|_{L^2(\Omega)}+ \|f\|_{L^2(Q)}\right).
\end{equation}
\end{theorem}
\begin{proof}
We prove the result in several steps. We shall use semigroups theory.\medskip

{\bf Step 1:}
Consider the bilinear form $\mathbb E:H_0^1(\Omega)\times H_0^1(\Omega)\to \mathbb R$ on $L^2(\Omega)$ given by
\begin{align}
\mathbb E(u,v)\coloneqq \int_{\Omega}\nabla u\cdot\nabla v\;\dx+\mathcal F(u,v).
\end{align}
Using the embedding \eqref{H1-to-Hs},  it is easy to see that the form $\mathbb E$ is continuous in the sense that there is a constant $C>0$ such that for every $u,v\in H_0^1(\Omega)$ we have,
\begin{align*}
|\mathbb E(u,v)|\le C\|u\|_{ H_0^1(\Omega)}\|v\|_{ H_0^1(\Omega)}.
\end{align*}
We claim that the form $\mathbb E$ is closed. Indeed, let $(u_n)_{n\in\N}$ be a sequence in $ H_0^1(\Omega)$ such that
\begin{align}\label{form-C}
\lim_{n, m\to\infty}\Big(\mathbb E(u_n-u_m,u_n-u_m)+\|u_n-u_m\|_{L^2(\Omega)}^2\Big)=0.
\end{align}
It follows from \eqref{form-C} that $(u_n)_{n\in \N}$ is a Cauchy sequence in the Hilbert space $ H_0^1(\Omega)$.  Therefore, there is an $u\in H_0^1(\Omega)$ such that $u_n\to u$ in $H_0^1(\Omega)$ as $n\to\infty$.  Hence, $u_n\to u$ in $H_0^s(\Omega)$ as $n\to\infty$, by using the continuous embedding \eqref{H1-to-Hs}. This implies that
\begin{align*}
\lim_{n\to\infty}\mathbb E(u_n-u,u_n-u)=0.
\end{align*}
Thus, the form $\mathbb E$ is closed and we have proved the claim.\medskip

It is easy to see that the form $\mathbb E$ is coercive in the sense that there is a constant $C>0$ such that
\begin{align*}
\mathbb E(u,u)\ge C\|u\|_{H_0^1(\Omega)}^2,\;\;\forall u\in H_0^1(\Omega).
\end{align*}

{\bf Step 2:} Let $\mathbb A$ be the selfadjoint operator on $L^2(\Omega)$ associated with $\mathbb E$ in the sense that
\begin{equation}\label{op-A}
\begin{cases}
D(\mathbb A)\coloneqq \Big\{u\in H_0^1(\Omega):\;\exists f\in L^2(\Omega),\;\mathbb E(u,v)=(f,v)_{L^2(\Omega)}\;\forall v\in H_0^1(\Omega)\Big\},\\
\mathbb Au=f.
\end{cases}
\end{equation}
Using an integration by parts argument and the results obtained in \cite{BDVV2020,Cl-Wa}, we can show that
\begin{equation}\label{op-A-2}
D(\mathbb A)\coloneqq \Big\{u\in H_0^1(\Omega):\; ( \mathscr{L}u)|_{\Omega}\in L^2(\Omega)\Big\},\;\;
\mathbb Au=( \mathscr{L}u)|_{\Omega}\mbox{ in }\,\Omega.
\end{equation}
We  have shown that the system \eqref{p1} can be rewritten as the following abstract Cauchy problem
\begin{equation}\label{ACP}
\begin{cases}
\phi_t+\mathbb A\phi=f \;&\mbox{ in }\; Q,\\
\phi(\cdot,0)=\phi_0&\mbox{ in }\;\Omega.
\end{cases}
\end{equation}

Since the form $\mathbb E$ is non-negative, continuous, closed and $H_0^1(\Omega)$ is dense in $L^2(\Omega)$, it follows that the operator $-\mathbb A$ generates a strongly continuous semigroup $(e^{-t\mathbb A})_{t\ge 0}$ on $L^2(\Omega)$. This implies that for every $f\in L^2(Q)$ the Cauchy problem \eqref{ACP}, hence \eqref{p1}, has a unique strong solution $\phi\in \mathbb U$ given for a.e. $x\in\Om$ and a.e. $t\in (0,T)$ by
\begin{align*}
\phi(x,t)= e^{-t\mathbb A}\phi_0(x)+\int_0^te^{-(t-\tau)\mathbb A}f(x,\tau)\;\mathrm{d}\tau.
\end{align*}

{\bf Step 3:} It remains to prove  \eqref{estimation1}.  First, taking $\zeta=\phi$ as a test function in \eqref{Eq-Def31}, integrating over $(0,T)$, and using Young's inequality, we get that for every $\varepsilon>0$,
\begin{align*}
& \frac{1}{2}\|\phi(\cdot,T)\|^2_{L^2(\Omega)}+ \int_0^T\dis \int_{\Omega}|\nabla \phi|^2 \dx\; \dt +\int_0^T\mathcal{F}(\phi,\phi) \dt\\
=&\frac{1}{2}\|\phi_{0}\|^2_{L^2(\Omega)}+\int_Qf\phi\, \dx\; \dt  \\
\leq & \frac{1}{2}\|\phi_{0}\|^2_{L^2(\Omega)}+ \frac{1}{2\varepsilon}\|f\|^2_{L^2(Q)}+ \frac{\varepsilon}{2}\|\phi\|^2_{L^2(Q)}\\
\leq & \frac{1}{2}\|\phi_{0}\|^2_{L^2(\Omega)}+ \frac{1}{2\varepsilon}\|f\|^2_{L^2(Q)}+ \frac{\varepsilon}{2}\|\phi\|^2_{L^2((0,T);H_0^1(\Omega))}.
 \end{align*}
Choosing $\varepsilon>0$ small enough, we can deduce that there is a constant $C>0$ such that
\begin{align}\label{INEE}
\|\phi(T)\|^2_{L^2(\Omega)}+ \|\phi\|^2_{L^2((0,T);H_0^1(\Omega))}
\leq  C\left(\|\phi_{0}\|^2_{L^2(\Omega)}+ \|f\|^2_{L^2(Q)}\right).
\end{align}
Second, it follows from \eqref{Eq-Def31} and \eqref{INEE} that there is a constant $C>0$ such that
\begin{align}\label{INEE-2}
\left| \int_0^T\langle \phi_t,\zeta\rangle_{H^{-1}(\Om),H_0^1(\Om)}\;dt \right|\le &C\left(\|\phi\|_{L^2((0,T);H_0^1(\Omega)}+\|f\|_{L^2(Q)}\right)\|\zeta\|_{L^2((0,T);H_0^1(\Omega))}\notag\\
\le &C\left(\|\phi_0\|_{L^2((\Omega)}+\|f\|_{L^2(Q)}\right)\|\zeta\|_{L^2((0,T);H_0^1(\Omega))}.
\end{align}
Dividing both sides of \eqref{INEE-2} by $\|\zeta\|_{L^2((0,T);H_0^1(\Omega)}$ and taking the supremum over all functions $\zeta\in L^2((0,T);H_0^1(\Omega))$, we get that
\begin{align}\label{INEE-3}
\|\phi_t\|_{L^2((0,T);H^{-1}(\Om))}\le C\left(\|\phi_0\|_{L^2((\Omega)}+\|f\|_{L^2(Q)}\right).
\end{align}
Combining \eqref{INEE}-\eqref{INEE-3} we get the estimate \eqref{estimation1} and the proof is finished.
\end{proof}

\begin{remark}
We mention that if in \eqref{p1} $\phi_0\in D(\mathbb A)$, then the strong solution $\phi$ becomes a classical solution, and hence, enjoys the following additional regularity: $\phi\in C([0,T];D(\mathbb A))\cap H^1((0,T);L^2(\Omega))$.  We refer to \cite{ABHN,BWZ3-1} and the references therein for more details on semigroups theory and abstract Cauchy problems.
\end{remark}

Next, we give further qualitative properties of the operator $\mathbb A$ and the semigroup $(e^{-t\mathbb A})_{t\ge 0}$ constructed above. Even if all these results will not be used in the present paper, they are interesting on their own and deserve to be known by the mathematics community working in the field.

\begin{proposition}
The operator $\mathbb A$ has a compact resolvent. Its eigenvalues form a non-decreasing sequence of real numbers $(\lambda_n)_{n\in\mathbb N}$ satisfying
\begin{align}\label{eigen}
0<\lambda_1\le \lambda_2\le\cdots\le\lambda_n\le\cdots\;\mbox{ and }\; \lim_{n\to\infty}\lambda_n=\infty.
\end{align}
The semigroup $(e^{-t\mathbb A})_{t\ge 0}$  is submarkovian and ultracontractive.
\end{proposition}

\begin{proof}
We prove the results in several steps.\medskip

{\bf Step 1}: Since the embedding $H_0^1(\Omega)\hookrightarrow L^2(\Omega)$ is compact, we have that the operator $\mathbb A$ has a compact resolvent. Therefore, its spectrum is composed with eigenvalues satisfying $0\le \lambda_1\le \lambda_2\le\cdots\le\lambda_n\le\cdots\;\mbox{ and }\; \lim_{n\to\infty}\lambda_n=\infty$.
Since the form $\mathbb E$ is coercive, it follows that the first eigenvalue $\lambda_1$  is strictly positive. Thus, $(\lambda_n)_{n\in\mathbb N}$ satisfies \eqref{eigen}.\medskip

{\bf Step 2:} We claim that the semigroup $(e^{-t\mathbb A})_{t\ge 0}$ is positivity-preserving in the sense
\begin{align}\label{PP}
0\le u\in L^2(\Omega)\;\mbox{ implies }\; e^{-t\mathbb A}u\ge 0,\;\forall\;t\ge 0.
\end{align}
The First Beurling-Deny criterion \cite[Theorem 1.3.1]{Dav} shows that \eqref{PP} is equivalent to 
\begin{align}\label{FBDC}
u\in H_0^1(\Om)\Rightarrow\; u^{+} \coloneqq \max\{u,0\}\in H_0^1(\Om)\;\mbox{ and }\;  \mathbb E(u^+,u^-)\le 0.
\end{align}

Indeed, let $u\in  H_0^1(\Omega)$ and set $u^{+} \coloneqq \max\{u,0\}$ and $u^{-} \coloneqq \max\{-u,0\}$. It follows from \cite[Chapter 1]{Tru} that $u^+$, $u^-\in  H_0^1(\Omega)$  and
\begin{align*}
\int_{\Omega}\nabla(u^+)\cdot\nabla(u^-)\;\dx=\int_{\Omega}(\nabla u)\chi_{\{u\ge 0\}} (\nabla u)\chi_{\{u\le 0\}}\;\dx=0.
\end{align*}
It has been shown in \cite{GW-CPDE}  that $\mathcal F(u^+,u^-)\le 0$. Thus,
\begin{align*}
\mathbb E(u^+,u^-)\le 0,\;\;\forall\;u\in H_0^1(\Omega).
\end{align*}
We have shown that  $(e^{-t\mathbb A})_{t\ge 0}$ is positivity-preserving.\medskip

{\bf Step 3:} We claim that $(e^{-t\mathbb A})_{t\ge 0}$ is $L^\infty$-contractive in the sense
\begin{align*}
\|e^{-t\mathbb A}u\|_{L^\infty(\Omega)}\le \|u\|_{L^\infty(\Omega)},\;\forall\;t\ge 0\;\mbox{ and } u\in L^2(\Omega)\cap L^\infty(\Omega)=L^\infty(\Omega).
\end{align*}
For this, let $u\in H_0^1(\Omega)$ be such that $u\ge 0$. It follows again from \cite[Chapter 1]{Tru} that $u\wedge 1\in  H_0^1(\Omega)$. A simple calculation gives 
\begin{align*}
\int_{\Omega}|\nabla (u\wedge 1)|^2\;\dx\le\int_{\Omega}|\nabla u|^2\;\dx.
\end{align*}
By \cite[Lemma 2.7]{War},  we have that $\mathcal F(u\wedge 1, u\wedge 1)\le\mathcal F(u,u)$. We have shown that
\begin{align*}
\mathbb E(u\wedge 1,u\wedge 1)\le \mathbb E(u,u),\;\;\forall\; 0\le u\in   H_0^1(\Omega).
\end{align*}
It follows from the second Beurling--Deny criterion \cite[Theorem 1.3.2]{Dav} that  $(e^{-t\mathbb A})_{t\ge 0}$ is $L^\infty$-contractive.
We have shown that the semigroup $(e^{-t\mathbb A})_{t\ge 0}$ is submarkovian. As a consequence,  $(e^{-t\mathbb A})_{t\ge 0}$ can be extended to consistent semigroups on $L^p(\Omega)$ ($1\le p\le\infty)$. Each semigroup is strongly continuous on $L^p(\Omega)$ if $1\le p<\infty$, and bounded analytic if $1<p<\infty$ (see e.g.  \cite[Chapter 1]{Dav} for more details).
\medskip

{\bf Step 4:} It remains to show that the semigroup is ultracontractive. Indeed, notice that we have the following continuous embedding:
\begin{align*}
 H_0^1(\Omega)\hookrightarrow L^{r}(\Omega)\;\mbox{ with } r=\begin{cases}  \frac{2N}{N-2}\;\;&\mbox{ if } N>2\\ 1\le r<\infty &\mbox{ if } N\le 2.\end{cases}
\end{align*}
That is, there is a constant $C>0$ such that for every $u\in H_0^1(\Omega)$,
\begin{align}\label{ult}
\|u\|_{ L^{r}(\Omega)}^2\le C\mathbb E(u,u).
\end{align}
Using the abstract results contained in \cite{Dav,Ouh}, the estimate \eqref{ult} is equivalent to the ultracontractivity of $(e^{-t\mathbb A})_{t\ge 0}$. More precisely, for every $1\le p\le q\le\infty$, the operator $e^{-t\mathbb A}$ maps $L^p(\Omega)$ into $L^q(\Omega)$ and there is a constant $C>0$ such that for every $t>0$  and $u\in L^p(\Omega)$,
\begin{align*}
\|e^{-t\mathbb A}u\|_{L^q(\Omega)}\le Ct^{-\frac N2\left(\frac 1p-\frac 1q\right)}\|u\|_{L^p(\Omega)}.
\end{align*}
The proof is finished.
\end{proof}

We mention that the results obtained for the  homogeneous problem \eqref{p1} are classical and they follow from semigroups theory associated with bilnear operators studied in \cite{Dav,Ouh} and the properties of the fractional Laplace operator investigated in \cite{BVDV,NPV,DRV,GW-CPDE,Ro-Sj,War} and their references.

Next, we consider the nonhomogeneous boundary-exterior-initial value problem \eqref{Eq:main_A1}, that is,
\begin{equation}\label{p1nh}
\begin{cases}
\phi_{t} + \mathscr{L}\phi =  0& \mbox{ in }\; Q,\\
\phi=u_1&\mbox{ on }\;\Gamma,\\
\phi = u_2 &\mbox{ in }\; \Sigma, \\
\phi(\cdot,0) = 0 &\mbox{ in }\; \Omega.
\end{cases}
\end{equation}

First, we consider smooth boundary-exterior data.

\begin{definition}\label{Dwsi}
Let $u_2\in H^1((0,T); H^1(\Omb)$ and  $u_1\in L^2((0,T);H^{1/2}(\pOm))$ be such that $u_2|_{\Gamma}=u_1$.  Let $\tilde u\in H^1((0,T);H^1(\RR^N)$ be such that $\tilde u=u_2$ in $\Sigma$.  A function $\phi\in L^2((0,T);H^1(\RR^N))\cap H^1((0,T);H^{-1}(\Omega))$ is said to be a weak solution of the system \eqref{p1nh} if $\phi-\tilde u\in L^2((0,T);H_0^1(\Omega))\cap H^1((0,T);H^{-1}(\Omega))$ and the identity
\begin{align}
\langle\phi_t,\zeta\rangle_{H^{-1}(\Omega),H_0^1(\Omega)}+\int_{\Omega}\nabla\phi\cdot\nabla\zeta\;\dx+\mathcal F(\phi,\zeta)=0
\end{align}
holds, for every $\zeta\in H_0^1(\Omega)$ and almost every $t\in (0,T)$.
\end{definition}

\begin{remark}
In Definition \ref{Dwsi}, we have assumed that the function $u_2$ (hence, the solution $\phi$) has a well-defined trace that coincides with $u_1\in L^2(\Gamma)$. In that case,  the condition on $\Gamma$ can be dropped and the system is still well-posed. That is,  \eqref{p1nh} becomes
\begin{equation*}
\begin{cases}
\phi_{t} + \mathscr{L}\phi =  0& \mbox{ in }\; Q,\\
\phi = u_2 &\mbox{ in }\; \Sigma, \\
\phi(\cdot,0) = 0 &\mbox{ in }\; \Omega.
\end{cases}
\end{equation*}
\end{remark}

We have the following existence result.

\begin{theorem}
Let  $u_2$ and $u_1$ be as in Definition \ref{Dwsi}. Then, the system \eqref{p1nh} has a unique weak solution $\phi\in L^2((0,T);H^1(\RR^N))\cap H^1((0,T);H^{-1}(\Omega))$ in the sense of Definition \ref{Dwsi}. In addition, there is a constant $C>0$ such that
\begin{align}\label{EST1}
\|\phi\|_{L^2((0,T);H^1(\RR^N))\cap H^1((0,T);H^{-1}(\Omega)}\le C\|u_2\|_{H^1((0,T); H^1(\Omb))}
\end{align}
\end{theorem}
\begin{proof}
We prove the result in several steps.\medskip

{\bf Step 1}: First, assume that $u_2$ does not depend on the time variable.
Let $\tilde u\in H^1(\RR^N)$ be the unique weak solution of the Dirichlet problem
\begin{equation}\label{HDP}
\begin{cases}
 \mathscr{L}\tilde u=0\;\;&\mbox{ in }\; \Omega,\\
\tilde u=u_2&\mbox{ in }\;\Omb.
\end{cases}
\end{equation}
That is, $\tilde u\in H^1(\RR^N)$, $\tilde u|_{\Omb}=u_2$,   and $\tilde u$ solves \eqref{HDP} in the sense that
\begin{align*}
\int_{\Omega}\nabla \tilde u\cdot \nabla v\;\dx+\mathcal F(\tilde u,v)=0\;\;\mbox{ for all }\;v\in H_0^1(\Omega),
\end{align*}
and there is a constant $C>0$ such that
\begin{align}\label{316}
\|\tilde u\|_{H^1(\RR^N)}\le  C\|u_2\|_{H^1(\Omb)}.
\end{align}
The proof of the existence, uniqueness of such a solution $\tilde u$ and the continuous dependence on the datum $u$ follow from Proposition \ref{proposi-33} by taking $f=0$ in \eqref{EDP}.\medskip

{\bf Step 2}: Second, assume that $u_2$ depends on both variables $(x,t)$ and satisfies the assumption of the theorem.  Let $\tilde u$ be the associated solution of \eqref{HDP}.
It follows from the above argument that $\tilde u\in H^1((0,T);H^1(\RR^N))$.
Let $\Phi\coloneqq \phi-\tilde u$. Then, it is clear that $\Phi|_{\Sigma}=0$. In addition, a simple calculation shows that
\begin{equation}\label{p1nh-2}
\begin{cases}
\Phi_{t} + \mathscr{L}\Phi =  -\tilde u_t& \mbox{ in }\; Q,\\
\Phi = 0 &\mbox{ in }\; \Sigma, \\
\Phi(\cdot,0) = 0 &\mbox{ in }\; \Omega.
\end{cases}
\end{equation}
Let
\begin{align*}
 \mathbb H \coloneqq L^2((0,T);H^1(\RR^N))\cap H^1((0,T);H^{-1}(\Omega)).
\end{align*}
Since  $\tilde{u}_t \in L^2((0,T);H^{1}(\RR^N))$,  using
Theorem~\ref{Thm1}, we get that there exists a unique $\Phi \in \mathbb{H}$ solving \eqref{p1nh-2}. Thus, the unique solution $\phi$ of \eqref{p1nh} is given by $\phi=\Phi+\tilde u$. \medskip

{\bf Step 3}: It remains to show \eqref{EST1}.
Firstly, since $\Phi=0$ in $\Sigma$ and $\Phi(\cdot,0)=0$ in $\Omega$, it follows from \eqref{estimation1} that there is a constant $C>0$ such that
  \begin{align}\label{A1}
  \|\Phi\|_{\mathbb H}\le C\|\tilde u_t\|_{L^2((0,T);H^{1}(\RR^N))}.
  \end{align}	
Secondly, it follows from \eqref{316} that there is a constant $C>0$ such that
\begin{align}\label{A2}
  \|\tilde{u}\|_{L^2((0,T);H^{1}(\RR^N))}  \le C\|u_2\|_{L^2((0,T);H^{1}(\Omb))}.
\end{align}
Thirdly, combining \eqref{A1}-\eqref{A2} and using \eqref{316}, we get that there is a constant $C>0$ such that
\begin{align*}  
\|\phi\|_{\mathbb H}&=\|\Phi+\tilde u\|_{\mathbb H}\le \|\Phi\|_{\mathbb H}+\|\tilde u\|_{\mathbb H}\\
&\le C\left(\|\tilde u_t\|_{L^2((0,T);H^{1}(\RR^N))}+\|u_2|_{L^2((0,T);H^{1}(\Omb))}\right)\notag\\
&\le C\left(\|u_t\|_{L^2((0,T);H^{1}(\Omb))}+\|u_2\|_{L^2((0,T);H^{1}(\Omb))}\right)\notag\\
&=C\|u_2\|_{H^1((0,T);H^{1}(\Omb))}.
\end{align*}
We have shown \eqref{EST1} and  the proof is finished.
\end{proof}

Finally, we consider singular boundary-exterior data.
It is worthwhile noticing that  since we are considering data $u_1\in L^2(\Gamma)$ and $u_2\in L^2(\Sigma)$, the system \eqref{p1nh} cannot have weak solutions in the sense of Definition \ref{Dwsi}. For this reason, we need to introduce the notion of very-weak solutions as in the elliptic case.

\begin{definition}\label{def:VWS}
Let $u_1\in L^2(\Gamma)$ and $u_2\in L^2(\Sigma)$.  A function $\phi\in L^2((0,T)\times\RR^N)$ is said to be a very weak-solution (or a solution by transposition) to the system \eqref{p1nh} if the identity
\begin{align}\label{VWS}
\int_{Q}\phi\Big(-\varphi_t+ \mathscr{L}\varphi\Big)\;\dx\;\dt
= -\int_{\Gamma} u_1\partial_\nu \varphi \;\mathrm{d}\sigma\, \dt -\int_\Sigma u_2\mathcal{N}_s\varphi \;\dx\, \dt,
\end{align}
holds, for every $\varphi\in L^2((0,T),\mathbb V)\cap H^1((0,T); L^2(\Omega))$ with $\varphi(\cdot,T)=0$ a.e. in $\Omega$, where we recall that $\mathbb V\coloneqq D(\mathbb A)$ is the space defined in \eqref{Rea}.
\end{definition}

As in the elliptic case,  Definition \ref{def:VWS} of very-weak solutions makes sense if every function $\varphi\in L^2((0,T),\mathbb V)\cap H^1((0,T); L^2(\Omega))$ satisfies $\partial_\nu\varphi \in L^2(\Gamma)$, and $\mathcal N_s\varphi\in L^2(\Sigma)$.

We have the following result.

\begin{theorem}\label{thm:VWS_Exist_1}
Let $0<s\le 3/4$, $u_1\in L^2(\Gamma)$, and $u_2\in L^2(\Sigma)$. Then, there exists a unique very-weak solution $\phi\in L^2((0,T)\times\RR^N)$ to \eqref{p1nh} according to Definition~\ref{def:VWS}, and there is a constant $C>0$ such that
\begin{align}\label{VWS_EST_1}
\|\phi\|_{L^2((0,T)\times\RR^N)} \le C \left(\|u_1\|_{L^2(\Gamma)}+\|u_2\|_{L^2(\Sigma)}\right).
\end{align}
Moreover, if $u_1$ and $u_2$ are as in Definition \ref{Dwsi}, then the following assertions hold.
\begin{enumerate}
\item Every weak solution of  \eqref{p1nh} is also a very-weak solution.
\item Every very-weak solution of  \eqref{p1nh} that belongs to $\mathbb H$ is also a weak solution.
\end{enumerate}
 \end{theorem}

\begin{proof}
 We prove the result in several steps.\medskip

{\bf Step 1:} For a given $\eta \in L^2(Q)$, we consider the following dual problem associated to \eqref{p1nh}
\begin{equation}\label{Eq:Sdual}
\begin{cases}
-\varphi_t+ \mathscr{L}\varphi = \eta \quad &\mbox{ in } Q, \\
\varphi=0 &\mbox{ in }\;\Gamma,\\
\varphi = 0           \quad &\mbox{ in } \Sigma ,  \\
\varphi(T,\cdot)  = 0 \quad &\mbox{ in } \Om .
 \end{cases}
 \end{equation}
Using semigroups theory  as in the proof of Theorem~\ref{Thm1}, we can deduce that for the every $\eta\in L^2(Q)$ the system \eqref{Eq:Sdual} has a unique weak solution
$\varphi\in L^2((0,T),\mathbb V)\cap H^1((0,T); L^2(\Omega))$ given for a.e. $t\in (0,T)$ by
\begin{align*}
\varphi(\cdot,t)=\int_t^T e^{-(T-\tau)\mathbb A}\eta(\cdot,\tau)\;\mathrm{d}\tau\;\;\mbox{ in }\;\Omega.
\end{align*}

This implies that $\varphi_{t}\in L^2(Q)$. Since $\varphi\in L^2((0,T),\mathbb V)\cap H^1((0,T); L^2(\Omega))$, we have that  $\mathcal{N}_s \varphi \in L^2(\Sigma)$ and $\partial_\nu\varphi\in L^2(\Gamma)$ under the assumption $0<s\le 3/4$.  We define the mapping
 \begin{equation*}
\Lambda: L^2(Q) \rightarrow  L^2(\Gamma)\times L^2(\Sigma), \qquad \eta \mapsto \Lambda \eta \coloneqq (-\partial_\nu\varphi, -\mathcal{N}_s \varphi).
\end{equation*}
It is clear that $\Lambda$ is linear. In addition, using the continuous dependence on the data of solutions to \eqref{Eq:Sdual}, we get that there is a  constant $C>0$ such that
 \begin{align}\label{MJW}
\|\Lambda \eta \|_{L^2(\Sigma)\times L^2(\Gamma)} \le C\left( \|\mathcal{N}_s \varphi\|_{L^2((\Sigma)} +\|\partial_\nu\varphi\|_{L^2(\Gamma)}\right)
\leq C \|\varphi\|_{L^2((0,T);\mathbb V)}\le C\|\eta\|_{L^2(Q)}.
\end{align}
We have shown that $\Lambda$ is also continuous.

Next, let $\phi \coloneqq \Lambda^* u$ in $Q$ and $\phi:=u_2$ in $\Sigma$.  Calculating we get the following:
\begin{align}\label{EQL}
\int_Q \phi \eta \; \dx\, \dt
      =& \int_Q \phi \left(-\varphi_t+ \mathscr{L}\varphi\right) \;\dx\, \dt
      = \int_Q (\Lambda^* u) \eta \; \dx\, \dt\notag\\
      = &-\int_\Sigma u_2 \mathcal{N}_s \varphi \; \dx\, \dt -\int_{\Gamma}u_1\partial_\nu\varphi\;\mathrm{d}\sigma\;dt.
 \end{align}
We have constructed a function $\phi \in L^2((0,T);L^2(\RR^N))$ that solves \eqref{p1nh} in the very weak-sense.

Next, we show uniqueness. Assume that \eqref{p1nh} has two very-weak solutions $\phi_{1}$ and $\phi_{2}$ with  the same boundary datum $u_1$ and the same exterior datum $u_2$. It follows from the definition  that
\begin{align}\label{sol}
\int_{Q}\left(\phi_{1} - \phi_{2}\right)\left(-\varphi_t+ \mathscr{L}\varphi\right)\;\dx\; \dt = 0,
\end{align}
for every $\varphi \in L^2((0,T),\mathbb V)\cap H^1((0,T); L^2(\Omega))$ with $\varphi(T,\cdot) = 0$ a.e.~in $\Omega$. Since for every $\eta\in L^2(Q)$ the system \eqref{Eq:Sdual} has a unique weak solution $\varphi$, it follows from \eqref{sol} and \eqref{Eq:Sdual} that
\begin{align*}
\int_{Q}\left(\phi_{1} - \phi_{2}\right)\eta\;\dx\;\dt = 0,
\end{align*}
for every $\eta\in L^2(Q)$. It follows from the fundamental lemma of the calculus of variation that  $\phi_{1} = \phi_{2}$ a.e. in $Q$. Since $\phi_{1} = \phi_{2}=u_2$ a.e. in $\Sigma$, we can conclude that $\phi_1=\phi_2$ a.e. in $(0,T)\times\RR^N$, and we have shown uniqueness of very-weak solutions.\medskip

{\bf Step 2}: Next, we show the estimate \eqref{VWS_EST_1}. Using  \eqref{EQL} and  \eqref{MJW}, we get that there is a constant $C>0$ such that
\begin{align}\label{mjw}
\left|\int_Q \phi \eta \; \dx\; \dt\right|\le &C\left(\|u_2\|_{L^2(\Sigma)}\|\mathcal N_s\varphi\|_{L^2(\Sigma)}+\|u_1\|_{L^2(\Gamma)}\|\partial_\nu\varphi\|_{L^2(\Gamma)}\right)\notag\\
\le &C\left(\|u_2\|_{L^2(\Sigma)}+\|u_1\|_{L^2(\Gamma)}\right)
\left(\|\partial_\nu\varphi\|_{L^2(\Gamma)}+\|\mathcal N_s\varphi\|_{L^2(\Sigma)}\right)\notag\\
\le &C\left(\|u_2\|_{L^2(\Sigma)}+\|u_1\|_{L^2(\Gamma)}\right)\|\eta\|_{L^2(Q)}.
\end{align}
Dividing both sides of \eqref{mjw} by $\|\eta\|_{L^2(Q)}$ and taking the supremum over all $\eta\in L^2(Q)$, we get
\begin{align*}
\|\phi\|_{L^2(Q)}\le C\left(\|u_2\|_{L^2(\Sigma)}+\|u_1\|_{L^2(\Gamma)}\right).
\end{align*}
Since $\phi=u_2$ in $\Sigma$, it follows from the preceding estimate that
\begin{align*}
\|\phi\|_{L^2((0,T)\times \RR^N)}\le C\left(\|u_2\|_{L^2(\Sigma)}+\|u_1\|_{L^2(\Gamma)}\right)
\end{align*}
and we have shown the estimate \eqref{VWS_EST_1}.\medskip

{\bf Step 3:} Next, we prove the last two assertions of the theorem. For this, we assume that $u_1$ and $u_2$ are as in Definition \ref{Dwsi}.\medskip

(a) Let $\phi \in \mathbb{H} \coloneqq L^2(0,T);H^1(\RR^N))\cap H^1((0,T);H^{-1}(\Omega))$ be a weak solution to \eqref{p1nh}.  Notice that $\phi|_{\Sigma}\in  L^2(\Sigma)$ and $\phi|_{\Gamma}\in L^2(\Gamma)$. It follows from the definition that $\phi= u_2$ in $\Sigma$ and $\phi|_{\Gamma}=u_1$ on $\Gamma$. In particular, we have that
 \begin{align}\label{e1}
\langle  \phi_t(t,\cdot) , v \rangle_{H^{-1}(\Om),H_0^1(\Om)} + \int_{\Omega}\nabla \phi\cdot\nabla v\;\dx+\mathcal F(\phi,v)= 0,
 \end{align}
for every $v \in L^2((0,T);H^2(\Omega)\cap H_0^1(\Omega)) \cap H^1((0,T);L^2(\Om))$ with $v(\cdot,T)=0$ a.e. in $\Om$, and almost every $t \in (0,T)$. Since $v(\cdot,t)=0$ in $\Omb$, we have that
\begin{align}\label{e2}
\int_{\RR^N}\int_{\RR^N}&\frac{(\phi(t,x)-\phi(t,y))(v(t,x)-v(t,y))}{|x-y|^{N+2s}}\;\dx\; \dy\notag\\
&=\int\int_{\RR^{2N}\setminus(\RR^N\setminus\Omega)^2}\frac{(\phi(t,x)-\phi(t,y))(v(t,x)-v(t,y))}{|x-y|^{N+2s}}\;\;\dx\; \dy.
\end{align}
Using \eqref{e1}, \eqref{e2} and the integration by parts formulas \eqref{Int-Part}-\eqref{IBP-L} (notice that the test function $v$ is smooth enough), we get that
\begin{align*}
0=&\langle \phi_t(t,\cdot),v \rangle_{H^{-1}(\Om),H_0^1(\Om)} +\int_{\Omega}\nabla \phi\cdot\nabla v\;\dx  +\mathcal F(\phi,v)\\
=&\langle \phi_t(t,\cdot),v(t,\cdot) \rangle_{H^{-1}(\Om),H_0^1(\Om)} -\int_{\Omega}\phi\Delta v\;\dx+\int_{\pOm}\phi\partial_\nu v\;\mathrm{d}\sigma\\
&+ \int_{\Omega}\phi(-\Delta)^sv\;\dx+\int_{\RR^N\setminus\Omega}\phi\mathcal N_sv\;\dx\\
=&\langle \phi_t(t,\cdot),v(t,\cdot) \rangle_{H^{-1}(\Om),H_0^1(\Om)} +\int_{\Omega}\phi \mathscr{L}  v\;\dx+\int_{\pOm}u_1\partial_\nu v\;\mathrm{d}\sigma
+\int_{\RR^N\setminus\Omega}u_2\mathcal N_s v\;\dx.
\end{align*}
Integrating the previous identity by parts over $(0,T)$, we get that
\begin{align*}
-\int_0^T(\phi(\cdot,t),v_t(\cdot,t))_{L^2(\Om)} \;\dt &+\int_Q\phi \mathscr{L}  v\;\dx\;\dt +\int_{\Gamma}u\partial_\nu v\;\mathrm{d}\sigma \dt+\int_{\Sigma}u\mathcal N_s v\;\dx\; \dt=0.
\end{align*}
Since $\phi$, $v_\in L^2(Q)$, it follows from the preceding identity that
 \begin{align}\label{JJ}
  \int_{Q}\phi\Big(-v_t+ \mathscr{L}v\Big)\;\dx\; \dt=-\int_{\Gamma}u\partial_\nu v\;\mathrm{d}\sigma \; \dt-\int_{\Sigma}u\mathcal N_sv\;\dx\; \dt
 \end{align}
  for every $v \in L^2((0,T);H^2(\Omega)\cap H_0^1(\Omega)) \cap H^1((0,T);L^2(\Om))$ with $v(\cdot,T)=0$ a.e. in $\Om$.  Since $L^2((0,T);H^2(\Omega)\cap H_0^1(\Omega)) \cap H^1((0,T);L^2(\Om))$ is dense in $ L^2((0,T);\mathbb V) \cap H^1((0,T);L^2(\Om))$, it follows that \eqref{JJ} remains true for every $ v\in L^2((0,T);\mathbb V) \cap H^1((0,T);L^2(\Om))$ with $v(\cdot,T)=0$ a.e. in $\Om$. Thus, $\phi$ is a very-weak solution of \eqref{p1nh}. \medskip

 (b) Let $\phi$ be a very-weak solution to \eqref{p1nh} and assume that $\phi \in \mathbb{H}$.
Then $\phi=u_2$ in $\Sigma$ and $\phi|_{\Gamma}=u_1$ on $\Gamma$. Let  $\tilde{u} \in H^1((0,T);H^{1}(\RR^N))\cap H^1((0,T);L^2(\mathbb R^N))$ be such that $\tilde u|_{\Sigma}=u_2$. Then, clearly $\phi-\tilde{u} \in
 \mathbb{H}$.  As $\phi$ is a very-weak solution to \eqref{p1nh} and $0<s\le 3/4$, it follows from
 Definition~\ref{def:VWS} that for every $v \in L^2((0,T);\mathbb V) \cap H^1((0,T);L^2(\Om))$, we have
 \begin{align}\label{e3}
  \int_{Q}\phi(- v_t + \mathscr{L} v)\;\dx
  =-\int_{\Gamma}u_1\partial_\nu v\;\mathrm{d}\sigma\;\dt-\int_{\Sigma}u_2\mathcal N_sv\;\dx\;\dt.
 \end{align}
 Since $\phi \in \mathbb{H}$, $v = 0$ in $\Sigma$ and $v = 0$ on $\Gamma$, using the integration by parts formulas
 \eqref{Int-Part}-\eqref{IBP-L} and a density argument (taking first $v \in L^2((0,T);H^2(\Om)\cap H_0^1(\Om)) \cap H^1((0,T);L^2(\Om))$ and then using a density argument as above), we can deduce that
 \begin{align}\label{e4}
\int_0^T& \langle \phi_t(\cdot,t), (\cdot,t)v \rangle_{H^{-1}(\Om),H_0^1(\Om)}\;\dt +\int_Q\nabla \phi\cdot\nabla v\;\dx\; \dt+\int_0^T \int_{\RR^N}\mathcal F(\phi,v)\;\dt\notag\\
&=\int_{Q}\phi\Big(-v_t + \mathscr{L}v\Big)\;\dx\; \dt+\int_{\Gamma}\phi\partial_\nu v\;\mathrm{d}\sigma\;\dt+\int_{\Sigma}\phi\mathcal N_sv\;\dx\; \dt\notag\\
&=\int_{Q}\phi\Big( -v_t + \mathscr{L}v\Big)\;\dx\; \dt+\int_{\Gamma}u_1\partial_\nu v\;\mathrm{d}\sigma\;\dt+\int_{\Sigma}u_2\mathcal N_sv\;\dx\; \dt.
\end{align}
It follows from \eqref{e3}- \eqref{e4} that for every
$v \in L^2((0,T);\mathbb V) \cap H^1((0,T);L^2(\Om))$ we have
 \begin{equation}\label{eq:e5}
    \int_0^T\langle \phi_t(t,\cdot), v(t,\cdot) \rangle\;\dt  +
    \int_0^T\left(\int_{\Omega}\nabla \phi\cdot\nabla v\;\dx+\mathcal F(\phi,v)\right)\;\dt= 0.
 \end{equation}
 Since $\mathbb V$ is dense in $H_0^{1}(\Om)$ and $L^2(\Om)$ is dense in
 $H^{-1}(\Om)$, it follows that \eqref{eq:e5} remains true for every
$v\in L^2((0,T);H_0^{1}(\Om)) \cap H^1((0,T);H^{-1}(\Om))$ with $v(\cdot,T)=0$ a.e. in $\Om$. Notice that $v(\cdot,t) \in H^{1}_0(\Om)$ for a.e.
 $t \in (0,T]$. As a result, we
 have that the following pointwise formulation
 \begin{equation}\label{eq:e6}
    \langle \phi_t(\cdot,t), v \rangle_{H^{-1}(\Om),H_0^1(\Om)}\;  +
    \int_{\Omega}\nabla \phi\cdot\nabla v\;\dx +\mathcal F(\phi,v)= 0
 \end{equation}
 holds for every $v \in H_0^{1}(\Om)$ and a.e. $t\in (0,T)$. We have shown that $\phi$ is the
 unique weak solution of  \eqref{p1nh} according to Definition~\ref{Dwsi}. The
 proof is finished.
 \end{proof}

 To conclude this section we discuss the notion of very-weak solutions to the system \eqref{p1nh} by taking into account the value $\phi(\cdot,T)$. 

 \begin{definition}\label{def:VWS-2}
A function $\phi\in L^2((0,T)\times\RR^N)$ is said to be a a very weak-solution  (or a solution by transposition) to the system \eqref{p1nh} if the  identity
\begin{align}\label{VWS-2}
\int_{Q}\phi\Big(-\varphi_t+ \mathscr{L}\varphi\Big)\;\dx\;\dt
=&\langle \phi(\cdot,T),\varphi(\cdot,T)\rangle_{H^{-1}(\Om),H_0^1(\Om)}\notag\\
&-\int_{\Gamma} u_1\partial_\nu \varphi \;\mathrm{d}\sigma\, \dt -\int_\Sigma u_2\mathcal{N}_s\varphi \;\dx\, \dt,
\end{align}
holds, for every $\varphi\in C([0,T],\mathbb V)\cap H^1((0,T); L^2(\Omega))$.
\end{definition}

The proof of the following result follows exactly as the proof of Theorem \ref{thm:VWS_Exist_1}. We omit it for brevity.

\begin{proposition}
Let $0<s\le 3/4$. Let $u_1\in L^2(\Gamma)$ and $u_2\in L^2(\Sigma)$. Then, there exists a unique very-weak solution $\phi\in L^2((0,T)\times\RR^N)$ of \eqref{p1nh} in the sense of Definition~\ref{def:VWS-2}.
\end{proposition}

\begin{remark}\label{rem-318}
Let us discuss the regularity of very weak solutions in the sense of Definition~\ref{def:VWS-2}.  Let $\mathbb V^\star$ denotes the dual of $\mathbb V$ with respect to the pivot space $L^2(\Omega)$.
Noticing that the operator $ \mathscr{L}$ maps $L^2(\Omega)$ into $\mathbb V^\star\hookrightarrow H^{-2}(\Omega)$, and using the theory of evolution equations (see e.g. \cite[Chapter III, Section 9,6]{lions1971}), we can show that every  very-weak solution $\phi$ of \eqref{p1nh} in the sense of Definition~\ref{def:VWS-2} belongs to $L^2((0,T)\times\RR^N)\cap H^1((0,T),H^{-1}(\Omega))$. Thus, $\phi\in C([0,T],H^{-1}(\Omega))$ so that  \eqref{VWS-2} makes sense for every $\varphi\in C([0,T],\mathbb V)\cap H^1((0,T); L^2(\Omega))$. This fact will be used in Section \ref{sec-OCP} below.
\end{remark}

\section{Optimal control problems of mixed local-nonlocal PDEs}\label{Sec:opt}

The purpose of this section is to study the optimal control problem \eqref{Eq:main_A}. Recall that
$\mathcal{Z}_{D} \coloneqq L^2(\Gamma)\times L^{2}(\Sigma)$ is endowed with the norm given by
\begin{equation}\label{normzd}
\|(u_{1},u_{2})\|_{ \mathcal{Z}_{D}}=\Big(\|u_{1}\|^2_{L^2(\Gamma)}+\|u_{2}\|^2_{ L^{2}(\Sigma)}\Big)^{\frac 12}.
\end{equation}

We consider the following controlled equation:
\begin{equation}\label{eqcontrol}
\begin{cases}
\psi_{t} + \mathscr{L}\psi = 0 & \mbox{ in }\; Q,\\
\psi=u_1&\mbox{ on }\;\Gamma,\\
\psi = u_2 &\mbox{ in }\; \Sigma , \\
\psi(\cdot,0) = 0, &\mbox{ in }\; \Omega,
\end{cases}
\end{equation}
where the control $(u_1,u_2) \in \mathcal{Z}_{ad}$ with $\mathcal{Z}_{ad}$ being a closed and convex subset of $\mathcal{Z}_{D}$.

Under the assumption on the data, and $0<s\le 3/4$,  we know from Theorem \ref{thm:VWS_Exist_1} that there exists a $\psi\in L^2((0,T)\times \R^N)$ which is the unique very-weak solution to \eqref{Eq:main_A} in the sense of Definition \ref{def:VWS} or Definition \ref{def:VWS-2}.  On the other hand, it follows from Remark \ref{rem-318} that
$\psi(\cdot,T)$ exists and belongs to $H^{-1}(\Om)$.
Therefore, we can define on $\mathcal{Z}_{ad}$ the following two cost functions:
\begin{equation}\label{defJ1}
J_1(u_1,u_2)\coloneqq \frac{1}{2}\|\psi((u_{1},u_{2}))-z_d^1\|_{L^2(Q)}^2 + \frac{\beta}{2}\|(u_{1},u_{2})\|^{2}_{\mathcal{Z} _{D}}
\end{equation}
and
\begin{equation}\label{defJ2}
J_2(u_1,u_2)\coloneqq \frac{1}{2}\|\psi(T;(u_{1},u_{2}))-z_d^2\|_{H^{-1}(\Om)}^2 + \frac{\beta}{2}\|(u_{1},u_{2})\|^{2}_{\mathcal{Z} _{D}},
\end{equation}
where  $\beta>0$ is a real number,  $z_d^1\in L^2(Q) $,  $z_d^2\in H^{-1}(\Om)$, $\psi:=\psi(u_1,u_2)$ is the unique very-weak solution of \eqref{eqcontrol},
and
$$
\|\phi\|^2_{H^{-1}(\Omega)}=\langle (-\Delta_D)^{-1}\phi,\phi\rangle_{H^1_0(\Omega),H^{-1}(\Omega)}.
$$
Here, $(-\Delta_D)^{-1}\phi=\varrho$ with $\varrho$ the unique solution of the Dirichlet problem
\begin{align*}
-\Delta \varrho=\phi\;\mbox{  in }\; \Omega\;\mbox{ and }\; \varrho = 0\;\mbox{ on }\; \partial \Omega.
\end{align*}

 We are interested in the following minimization problems:
\begin{equation}\label{pbcontrol1}
\min_{ (v_1,v_2)\in \mathcal{Z}_{ad}} J_i((v_1,v_2)),\, i=1,2.
\end{equation}

\subsection{The first optimal control problem}\label{First-OCP}

In this section we consider the minimization problem \eqref{pbcontrol1}-\eqref{defJ1} with the functional $J_1$.
We have the following existence result of optimal solutions.

\begin{proposition}\label{prop1N}
Let $0<s\le 3/4$, $u_1\in L^2(\Gamma)$, and $u_2\in L^2(\Sigma)$.	Let $\mathcal{Z}_{ad}$ be a closed and convex subset of $\mathcal{Z}_{D}$, and  let $\psi=\psi(u_1,u_2)$ satisfy \eqref{eqcontrol} in the very-weak sense. Then there exists a unique  control $(u_1^\star,u_2^\star)\in \mathcal{Z}_{ad}$ solution of \begin{equation}\label{pbcontrol11}
\inf_{ (v_1,v_2)\in \mathcal{Z}_{D}} J_1((v_1,v_2)).
\end{equation}
\end{proposition}

\begin{proof}
Firstly,	observe that if  $(u_1,u_2)=(0,0)$, then  \eqref{eqcontrol} has the unique solution $\psi(0,0)=0$.

Secondly, a simple calculation gives
\begin{align*}
J_1(v_1,v_2)\coloneqq &\dis  \frac{1}{2}\|\psi((v_1,v_2))-z_d^1\|_{L^2(Q)}^2 + \frac{\beta}{2}\|(v_1,v_2)\|^{2}_{\mathcal{Z} _{D}}
\\
=&\dis\frac 12 \|\psi(v_1,v_2)\|^2_{L^2(Q)} -\dis \int_Q \psi(v_1,v_2)\, z_d^1\;\dx\;\dt
\dis +\frac 12 \|z_d^1\|^2_{L^2(Q)}+\frac{\beta}{2}\|(v_1,v_2)\|^{2}_{\mathcal{Z} _{D}}
\\
= &\pi((v_1,v_2),(v_1,v_2))-L((v_1,v_2))+
\|z_d^1\|^2_{L^2(Q)}
\end{align*}
where
\begin{align*} 
  \pi((u_1,u_2),(v_1,v_2)) \coloneqq \dis\frac 12 \int_Q \psi(u_1,u_2)\,\psi(v_1,v_2) \;\dx\;\dt
 + \frac{\beta}{2}\int_{\Gamma}v_1\, u_1 \;\dx\;\dt+ \frac{\beta}{2}\int_{\Sigma}v_2\, u_2 \;\dx\;\dt,
  \end{align*}
and
  \begin{equation*} 
  L(v_1,v_2) \coloneqq \dis \int_Q \psi(v_1,v_2)\,z_d^1\;\dx\;\dt.
\end{equation*}

It is clear that $\pi((\cdot,\cdot),(\cdot,\cdot))$ is a bilinear  and symmetric functional.
\begin{enumerate}
  \item We claim that  $\pi((\cdot,\cdot),(\cdot,\cdot))$ is continuous on $\mathcal{Z}_{ad}$. Indeed,  let $(u_1,u_2),(v_1,v_2)\in \mathcal{Z}_{ad}$. Using \eqref{VWS_EST_1}, we get that there is a constant $C>0$ such that
  \begin{align*}
  |\pi((u_1,u_2),(v_1,v_2))| \leq & \frac 12\|\psi(u_1,u_2)\|_{L^2(Q)}\|\psi(v_1,v_2)\|_{L^2(Q)}\\
  &+\dis \frac{\beta}{2}\|v_1\|_{L^2(\Gamma)}\, \|u_1\|_{L^2(\Gamma)}+\dis \frac{\beta}{2}\|v_2\|_{L^2(\Sigma)}\, \|u_2\|_{L^2(\Sigma)}\\
  \leq& C \left(\|u_1\|^2_{L^2(\Gamma)}+\|u_2\|^2_{L^2(\Sigma)}\right)^{1/2}
  \left(\|v_1\|^2_{L^2(\Gamma)}+\|v_2\|^2_{L^2(\Sigma)}\right)^{1/2}\\
  &+\dis  \frac{\beta}{2}\left(\|v_1\|^2_{L^2(\Gamma)}+\|v_2\|^2_{L^2(\Sigma)}\right)^{1/2}
  \left(\|u_1\|^2_{L^2(\Gamma)}+\|u_2\|^2_{L^2(\Sigma)}\right)^{1/2}\\
  \leq&\dis \left(C+\frac{\beta}{2}\right)\|(v_1,v_2)\|_{\mathcal{Z} _{D}}\|(u_1,u_2)\|_{\mathcal{Z} _{D}},
  \end{align*}
  and the claim is proved.\medskip

  \item We claim that $\pi((\cdot,\cdot),(\cdot,\cdot))$ is coercive on $\mathcal{Z}_{ad}$. Indeed, for all  $(u_1,u_2),(v_1,v_2)\in \mathcal{Z}_{ad}$, we have
  \begin{align*}
  \pi((u_1,u_2),(u_1,u_2)) = \frac 12 \|\psi(u_1,u_2)\|_{L^2(Q)}^2+\frac{\beta}{2}\|(v_1,v_2)\|^{2}_{\mathcal{Z} _{D}}\geq \frac{\beta}{2}\|(v_1,v_2)\|^{2}_{\mathcal{Z} _{D}},
  \end{align*}
  and we have shown the coercivity.

  \item Finally, using \eqref{VWS_EST_1}, we get that there is a constant $C>0$ such that for all  $(v_1,v_2)\in \mathcal{Z}_{ad}$,
  \begin{align*}
  |L(v_1,v_2)| \leq  \|\psi(v_1,v_2))\|_{L^2(Q)}\|z_d^1\|_{L^2(Q)}
  \leq C \|z_d^1\|_{L^2(Q)}\|(v_1,v_2)\|^{2}_{\mathcal{Z} _{D}}.
  \end{align*}
We have shown that  the functional $L$ is linear and continuous on $\mathcal{Z}_{ad}$.
  \end{enumerate}
Using the abstract results in \cite[Chapter II, Section 1.2]{lions1971}, we can then deduce that there exists a unique  $(u_{1}^\star, u_{2}^\star)\in \mathcal{Z}_{ad}$  solution to \eqref{pbcontrol11}.  The proof is complete.
\end{proof}


Next, we characterize the optimality conditions.

\begin{theorem}\label{Thm6}
Let $0<s\le 3/4$ and $\mathbb{U}=L^2((0,T);H^1_0(\Omega))\cap H^1((0,T);H^{-1}(\Omega)).$ Let also $\mathcal{Z}_{ad}$ be a closed convex subspace of $ \mathcal{Z}_{D}$,  and $(u_{1}^{\star}, u_{2}^{\star})$ be the minimizer  \eqref{pbcontrol11} over $\mathcal{Z}_{ad}$. Then, there exist $p^{\star}$ and $\psi^\star$ such that the triplet $(\psi^{\star},p^{\star},(u_{1}^{\star}, u_{2}^{\star}))\in L^2((0,T)\times \R^N)\times \mathbb{U}\times \mathcal{Z}_{ad}$ satisfies the following optimality systems:
\begin{equation}\label{optrho1}
\begin{cases}
\psi^{\star}_{t} + \mathscr{L}\psi^{\star} = 0 & \mbox{ in }\; Q,\\
\psi^{\star} = u_1^\star &\mbox{ in }\; \Gamma , \\
\psi^{\star} = u_{2}^{\star} &\mbox{ in }\; \Sigma , \\
\psi^{\star}(\cdot,0) = 0 &\mbox{ in }\; \Omega,
\end{cases}
\end{equation}
and
\begin{equation}\label{optp1}
\begin{cases}
-p^{\star}_t + \mathscr{L} p^{\star} = z_d^1-\psi^{\star} &\mbox{ in } Q, \\
p^{\star} = 0 &\mbox{ in }\; \Sigma , \\
p^{\star}(\cdot,T)  = 0 &\mbox{ in }\; \Om,
\end{cases}
\end{equation}
and
\begin{equation}\label{opu1}
 \int_\Gamma\Big(\partial_\nu p^{\star}+\beta u_{1}^{\star}\Big)(v_1-u_{1}^{\star})\, \mathrm{d}\sigma\, \dt+
 \int_\Sigma\Big(\mathcal{N}_s p^{\star}+\beta u_{2}^{\star}\Big)(v_2-u_{2}^{\star})\, \dx\, \dt\geq 0 \;\;\;\forall (v_1,v_2)\in \mathcal{Z}_{ad}.
 \end{equation}
 In addition,
 \begin{align}\label{coroopu1}
 (u_{1}^{\star}, u_{2}^{\star})= \dis \mathbb P(-\beta^{-1}\partial_\nu p^{\star}, -\beta^{-1}\mathcal{N}_s p^{\star}) 
 \end{align}
 where $\mathbb P$ is the projection onto the set $\mathcal{Z}_{ad}$.
\end{theorem}	

\begin{proof}
 Let  $(u_{1}^\star, u_{2}^\star)\in \mathcal{Z}_{ad}$  be the unique solution of the minimization problem \eqref{pbcontrol11}. We denote by $\psi^\star\coloneqq \psi^\star(u_{1}^\star, u_{2}^\star)$ the associated state so that, $\psi^\star$ solves the system \eqref{optrho1} in the very-weak sense.

Using classical duality arguments we have that \eqref{optp1} is the dual system associated with \eqref{optrho1}.  Since $z_d^1-\psi^{\star} \in L^2(Q)$, it follows that  \eqref{optp1} has a unique weak solution $p^\star\in \mathbb U$.

To prove the last assertion \eqref{opu1}, we write the Euler Lagrange first order optimality condition that characterizes the optimal control $(u_{1}^\star, u_{2}^\star)$ as follows:
\begin{align}\label{JJJ}
\lim_{\lambda\to 0}\frac{J_1(u_1^{\star}+\lambda (v_1-u_1^{\star}),u_2^{\star}+\lambda (v_2-u_2^{\star}))-J_1(u_1^{\star},u_2^{\star})}{\lambda}\geq 0,\;\;\forall v\coloneqq (v_1,v_2)\in \mathcal{Z}_{ad}.
\end{align}
Recall that
\begin{align*}
&J_1(u_1^{\star}+\lambda (v_1-u_1^{\star}),u_2^{\star}+\lambda (v_2-u_2^{\star}))\\
=&\frac 12\|\psi(u_1^{\star}+\lambda (v_1-u_1^{\star}),u_2^{\star}+\lambda (v_2-u_2^{\star}))-z_d^1\|_{L^2(Q)}^2\\\
&+\frac{\beta}{2}\|u_1^{\star}+\lambda (v_1-u_1^{\star}),(u_2^{\star}+\lambda (v_2-u_2^{\star}))\|_{\mathcal Z_D}^2,
\end{align*}
where $\psi^\lambda\coloneqq \psi(u_1^{\star}+\lambda (v_1-u_1^{\star}),(u_2^{\star}+\lambda (v_2-u_2^{\star}))$ is the unique very-weak solution of the system
\begin{equation}\label{solz1-1}
\begin{cases}
\psi_t^\lambda+\mathscr{L} \psi^\lambda = 0 & \mbox{in } Q, \\
\psi^\lambda =u_1^\star+\lambda( v_1- u_{1}^{\star}) & \mbox{in } \Gamma ,  \\
\psi^\lambda=u_2^\star+\lambda(v_2-u_{2}^{\star})  & \mbox{in } \Sigma ,  \\
\psi^\lambda(\cdot,0) = 0 & \mbox{in } \Om.
\end{cases}
 \end{equation}
Using the linearity of the system and the uniqueness of very-weak solutions, we get that
\begin{align}\label{US}
\psi^\lambda
=\psi^\star(u_1^{\star},u_2^{\star}) +\lambda\psi(v_1-u_1^{\star},v_2-u_2^{\star})=\psi^\star+\lambda\psi,
\end{align}
where $\psi$ is the unique very-weak solution of
\begin{equation} \label{solz1}
\begin{cases}
\psi_t+\mathscr{L} \psi = 0 & \mbox{in } Q, \\
\psi = v_1- u_{1}^{\star} & \mbox{in } \Gamma ,  \\
\psi= v_2-u_{2}^{\star}  & \mbox{in } \Sigma ,  \\
\psi(\cdot,0) = 0 & \mbox{in } \Om.
\end{cases}
 \end{equation}
It follows from \eqref{JJJ} and  \eqref{US} that
 \begin{align}\label{US2}
0\le &J_1(u_1^{\star}+\lambda (v_1-u_1^{\star}),u_2^{\star}+\lambda (v_2-u_2^{\star}))\notag\\
=&\frac 12 \|\psi^\star\|_{L^2(Q)}^2+\frac 12\lambda^2\|\psi\|_{L^2(Q)}^2+\frac 12\|z_d^1\|_{L^2(Q)}^2\notag\\
&+\lambda\int_{Q}\psi^\star\psi\;\dx\; \dt-\int_{Q}\psi^\star z_d^1\;\dx\; \dt
-\lambda\int_{Q}\psi z_d^1\;\dx\; \dt+\frac{\beta}{2}\|(u_1^\star,u_2^\star)\|_{\mathcal Z_D}^2\notag\\
 &+\frac{\lambda^2\beta}{2}\|(v_1-u_1^\star,v_2-u_2^\star)\|_{\mathcal Z_D}^2
 +\lambda\beta\int_{\Gamma}u_1^\star(v_1-u_1^\star)\;\mathrm{d}\sigma \dt \notag\\
 &+\lambda\beta\int_{\Sigma}u_2^\star(v_2-u_2^\star)\;\dx\;\dt.
 \end{align}
 It follows from \eqref{US2}  that
 \begin{align}\label{US2-2}
0\le & \frac{J_1(u_1^{\star}+\lambda (v_1-u_1^{\star}),u_2^{\star}+\lambda (v_2-u_2^{\star}))-J_1(u_1^{\star},u_2^{\star})}{\lambda}\notag\\
=&\frac 12\lambda\|\psi\|_{L^2(Q)}^2+\int_{Q}\psi^\star\psi\;\dx\;\dt+\beta\int_{\Sigma}u_2^\star(v_2-u_2^\star)\;\mathrm{d}\sigma \dt
-\int_{Q}\psi z_d^1\;\dx\; \dt\notag\\
& +\frac{\lambda\beta}{2}\|(v_1-u_1^\star,v_2-u_2^\star)\|_{\mathcal Z_D}^2+\beta\int_{\Gamma}u_1^\star(v_1-u_1^\star)\;\mathrm{d}\sigma \dt .
 \end{align}
Taking the limit of \eqref{US2-2} as $\lambda\downarrow 0$, we obtain
 \begin{align*}
 \int_{Q}\psi^\star \psi\;\dx\;\dt -\int_{Q}\psi z_d^1\;\dx\;\dt
+\beta\int_{\Gamma}u_1^\star(v_1-u_1^\star)\;\mathrm{d}\sigma \dt +\beta\int_{\Sigma}u_2^\star(v_2-u_2^\star)\;\mathrm{d}\sigma \dt\ge 0.
 \end{align*}
 That is, for all $(v_1,v_2)\in \mathcal{Z}_{ad}$, we have
\begin{align}\label{euler11}
 \int_Q \psi\Big(\psi^{\star}-z_d\Big)
\dx\, \dt+\beta\int_\Gamma u_{1}^{\star}(v_1-u_{1}^{\star})\, \;\;\mathrm{d}\sigma \dt+\beta\int_\Sigma u_{2}^{\star}(v_2-u_{2}^{\star})\, \dx\, \dt\geq 0.
\end{align}

Next, taking $p^{\star}$ (the solution of \eqref{optp1}) as a test function in the definition of very-weak solutions to \eqref{solz1} we get that
\begin{align}\label{euler21}
\int_Q\psi\Big(z_d^1-\psi^{\star}\Big)\;\dx\;\dt+\int_{\Gamma}(v_1-u^\star)\partial_\nu p^\star\;\;\mathrm{d}\sigma \dt+\int_{\Sigma}(v_2-u_2^\star)\mathcal N_s p^\star\;\dx\;\dt=0.
\end{align}
Combining \eqref{euler11}-\eqref{euler21}, we get \eqref{opu1}. The justification of \eqref{coroopu1} is classical and the proof is finished.
\end{proof}

\begin{remark}\label{coroThm6}
Let $0<s\le 3/4$, $\mathcal{Z}_{ad}= \mathcal{Z}_{D}$,  and $(u_{1}^{\star}, u_{2}^{\star})$ be the minimizer  of  \eqref{pbcontrol11} over $\mathcal{Z}_{ad}$.  It follows from \eqref{coroopu1} and Step 1 in the proof of Theorem \ref{th-36} that the regularity of $u_{1}^{\star}$ can be improved.  More precisely, if $0<s\le 3/4$, then $u_{1}^{\star}$ belongs to $L^2((0,T);H^{1/2}(\pOm))$. 
\end{remark}

\subsection{The second optimal control problem}\label{sec-OCP}

Here we consider the minimization problem
\begin{equation}\label{pbcontrol1-2}
\min_{ (v_1,v_2)\in \mathcal{Z}_{ad}} J_2((v_1,v_2)),
\end{equation}
with the functional $J_2$ given by
\begin{equation}\label{defJ2-2}
J_2(u_1,u_2)\coloneqq \frac{1}{2}\|\psi(T;(u_{1},u_{2}))-z_d^2\|_{H^{-1}(\Om)}^2 + \frac{\beta}{2}\|(u_{1},u_{2})\|^{2}_{\mathcal{Z} _{D}},
\end{equation}
where  $\beta>0$ is a real number, $z_d^2\in H^{-1}(\Om)$, the state $\psi\coloneqq \psi(u_1,u_2)$ is the unique very-weak solution  of \eqref{eqcontrol}, and $\psi(T;(u_{1},u_{2}))=\psi(\cdot,T)$.\medskip

We have the following existence result of optimal solutions.

\begin{proposition}\label{thm:docexist}
Let $0<s\le 3/4$, $\mathcal{Z}_{ad}$ a closed and convex subset of $\mathcal{Z}_{D}$, $u_1\in L^2(\Gamma)$, $u_2\in L^2(\Sigma)$, and  let $\psi\coloneqq \psi(u_1,u_2)$ satisfy \eqref{eqcontrol}.
Then, there exists a unique solution $(u_{1}^{\star}, u_{2}^{\star})$ to the minimization problem \eqref{pbcontrol1-2}-\eqref{defJ2-2}.
\end{proposition}

\begin{proof}
Here, we use minimizing sequences. Since the functional $J_2 : \mathcal{Z}_{ad} \rightarrow \mathbb{R}$ is bounded from below by zero, it is possible to construct a minimizing sequence $\{(u_{1n},u_{2n})\}_{n\in\mathbb{N}}$ such that
\begin{equation}\label{control1}
 \lim_{n\to \infty}J_2((u_{1n},u_{2n}))=\inf_{(v_{1},v_{2})\in \mathcal{Z}_{ad}}J_2((v_{1},v_{2})). 
 \end{equation}
We denote by $\psi_{n} \coloneqq \psi_n((u_{1n},u_{2n}))$ the state associated with the control $(u_{1n},u_{2n})$. Then, for each $n\in\NN$, we have that $\psi_n(u_{1n},u_{2n})$ is the unique very-weak solution of
\begin{equation}\label{eqcontroln}
\begin{cases}
(\psi_{n})_{t} + \mathscr{L}\psi_{n} = 0 & \mbox{ in }\; Q,\\
\psi_{n}= u_{1n}&\mbox{ on }\;\Gamma,\\
\psi_{n} = u_{2n} &\mbox{ in }\; \Sigma , \\
\psi_{n}(\cdot,0) = 0, &\mbox{ in }\; \Omega.
\end{cases}
\end{equation}
It  follows from \eqref{control1}, the structure of the cost function given by \eqref{defJ2-2}, and the definition of the norm on  $\mathcal{Z}_{D}$ given by \eqref{normzd} that, there exists a constant $C>0$ independent of $n$ such that,
\begin{align}
\|u_{1n}\|_{L^2(\Gamma)}&\leq  C,\label{estun1}\\
\|u_{2n}\|_{ L^{2}(\Sigma)} &\leq C,\label{estun2}\\
\|\psi_{n}(\cdot,T)\|_{H^{-1}(\Omega)}&\leq  C.\label{estrnT}
\end{align}
Since $\psi_n\in L^2((0,T)\times\R^N)$ is the unique very-weak solution solution of \eqref{eqcontroln}, it follows from   \eqref{VWS_EST_1}, \eqref{estun1} and \eqref{estun2} that,
\begin{align}\label{estpsin0}
\|\psi_n\|_{L^2(Q)} \leq\|\psi_n\|_{L^2((0,T)\times\RR^N)} \leq C.
\end{align}
It follows from \eqref{estun1},  \eqref{estun2},  \eqref{estrnT},  and \eqref{estpsin0} that there exist $(u_{1}^{\star},u_{2}^{\star}) \in L^2(\Gamma)\times L^2(\Sigma)$, $\psi_T\in H^{-1}(\Omega)$, and $\psi^{\star}\in  L^2((0,T)\times \R^N)$ such that,  as $n\to\infty$,
\begin{eqnarray}
u_{1n} \rightharpoonup u_{1}^{\star}&\mbox{ weakly in }& L^2(\Gamma),\label{cvun1}\\
u_{2n}\rightharpoonup u_{2}^{\star}&\mbox{ weakly in }& L^{2}(\Sigma),\label{cvun2}\\
\psi_n(\cdot, T)\rightharpoonup \psi_T &\mbox{ weakly in }& H^{-1}(\Omega),\label{cvrnT}\\
\psi_n \rightharpoonup \psi^{\star}&\mbox{ weakly in }&  L^2((0,T)\times \R^N).\label{cvrn}
\end{eqnarray}
Note that, from \eqref{estpsin0} and \eqref{cvrn} we have that,  as $n\to\infty$,
\begin{align}
\psi_n \rightharpoonup \psi^{\star} \mbox{ weakly in }  L^2(Q).\label{cvrnO}
\end{align}
Using \eqref{cvun1} and \eqref{cvun2},  we have that,  as $n\to\infty$,
$$
(u_{1n},u_{2n}) \rightharpoonup (u_{1}^{\star},u_{2}^{\star}) \mbox{ weakly in } \mathcal{Z}_{D}.
$$
Since $(u_{1n}, u_{2n})\in \mathcal{Z}_{ad}$ and $ \mathcal{Z}_{ad}$ is a closed subset of $\mathcal{Z}_{D}$, we have that
\begin{equation}
(u_{1}^{\star}, u_{2}^{\star})\in \mathcal{Z}_{ad}.\label{cvuad}
\end{equation}

It follows from the definition of very-weak solutions to the system \eqref{eqcontroln} that
\begin{align}\label{SA}
 \dis \int_Q\psi_n(-\phi_{t} + \mathscr{L}\phi)\, \dx\, \dt=&\dis \langle \psi_n(\cdot,T),\phi(\cdot,T)\rangle_{H^{-1}(\Om),H_0^{1}(\Om)} \notag\\
 &-\int_\Gamma u_{1n} \partial_\nu \phi \, \;\mathrm{d}\sigma \dt -\int_\Sigma u_{2n} \mathcal{N}_s\phi \,\dx\, \dt
\end{align}
for every $\phi\in L^2((0,T);\mathbb V)\cap H^1((0,T);L^2(\Omega))$  and $n\in\NN$.
Using all the above convergences, and taking the limit of \eqref{SA} as $n\to\infty$, we get that
\begin{align}\label{B1}
\int_Q \psi^{\star}\Big(-\phi_{t} + \mathscr{L}\phi\Big)\;\dx\; \dt= &\langle \psi_T,\phi(\cdot,T)\rangle_{H^{-1}(\Om),H_0^{1}(\Om)} - \int_\Gamma u_{1}^{\star} \partial_\nu \phi \,\;\mathrm{d}\sigma \dt \notag\\
&-\int_\Sigma u_{2}^{\star} \, \mathcal{N}_s\phi \,\dx\, \dt,
\end{align}
for every $\phi\in L^2((0,T);\mathbb V)\cap H^1((0,T);L^2(\Omega))$. In particular, we have that
\begin{align*}
\int_Q \psi^{\star}\Big(-\phi_{t} + \mathscr{L}\phi\Big)\; \;\mathrm{d}\sigma \dt = - \int_\Gamma u_{1}^{\star} \partial_\nu \phi \, \;\mathrm{d}\sigma \dt -\int_\Sigma u_{2}^{\star} \, \mathcal{N}_s\phi \,\dx\, \dt,
\end{align*}
$\phi\in L^2((0,T);\mathbb V)\cap H^1((0,T);L^2(\Omega))$ with $\phi(\cdot,T)=0$ a.e.  on $\Om$.  Since $\psi_n\in C([0,T];H^{-1}(\Omega))$ and $\psi_n(\cdot,0)=0$, we have that $\psi^\star(\cdot,0)=0$ in $\Omega$.
We have shown that $\psi^\star$ is a very-weak solution of
\begin{equation}\label{AA}
\begin{cases}
\psi^{\star}_{t} + \mathscr{L}\psi^{\star} = 0 & \mbox{ in }\; Q,\\
\psi^{\star} = u_{1}^{\star} &\mbox{ in }\; \Gamma , \\
\psi^{\star} = u_{2}^{\star} &\mbox{ in }\; \Sigma , \\
\psi^{\star}(\cdot,0) = 0&\mbox{ in }\; \Omega,
\end{cases}
\end{equation}
in the sense of Definition \ref{def:VWS}. It follows from Remark \ref{rem-318} that $\psi^\star$ enjoys the following additional regularity: $\psi^\star\in C([0,T];H^{-1}(\Omega)$. This implies that
\begin{align}\label{B2}
\int_Q \psi^{\star}\Big(-\phi_{t} + \mathscr{L}\phi\Big)\;\dx\;\dt= &\langle \psi^\star(\cdot,T),\phi(\cdot,T)\rangle_{H^{-1}(\Om),H_0^{1}(\Om)}\notag\\
&- \int_\Gamma u_{1}^{\star} \partial_\nu \phi \, \;\mathrm{d}\sigma \dt -\int_\Sigma u_{2}^{\star} \, \mathcal{N}_s\phi \,\dx\, \dt,
\end{align}
for every $\phi\in L^2((0,T);\mathbb V)\cap H^1((0,T);L^2(\Omega))$.
Combining \eqref{B1}-\eqref{B2}, we get
\begin{align*}
 \langle \psi_T,\phi(\cdot,T)\rangle_{H^{-1}(\Om),H_0^{1}(\Om)}= \langle \psi^\star(\cdot,T),\phi(\cdot,T)\rangle_{H^{-1}(\Om),H_0^{1}(\Om)}
\end{align*}
for every $\phi\in L^2((0,T);\mathbb V)\cap H^1((0,T);L^2(\Omega))$.  From which we can deduce that $\psi_T=\psi^\star(\cdot,T)$.\medskip

Next, since the functional $J_2$ is convex and lower semi-continuous, using \eqref{cvun1}, \eqref{cvun2},  \eqref{cvuad}, the fact that $\psi^\star(\cdot,T)=\psi_T$, and  \eqref{cvrnT}, we can deduce that
\begin{align*}
J_2((u_{1}^{\star}, u_{2}^{\star}))\leq \liminf_{ n\to +\infty}J_2((u_{1n}, u_{2n}))&=\lim_{ n\to +\infty}J_2((u_{1n}^{\star}, u_{2n}^{\star}))\\
&=\inf_{ (u_{1}, u_{2})\in \mathcal{Z}_{ad}} J_2((u_{1}, u_{2}))\le J_2((u_{1}^\star, u_{2}^\star)).
\end{align*}
We have shown that $(u_{1}^{\star}, u_{2}^{\star})$ is the optimal solution of \eqref{pbcontrol1-2}-\eqref{defJ2-2}.  The uniqueness is straightforward and follows directly from the strict convexity of  $J_2$. The proof is finished.
\end{proof}


The following result characterizes the optimality conditions.

\begin{theorem}\label{Thm5}
Let $0<s\le 3/4$ and $\mathbb{U}:=L^2(0,T;H^1_0(\Omega))\cap H^1((0,T);H^{-1}(\Omega)).$ Let $\mathcal{Z}_{ad}$ be a closed, convex subspace of $ \mathcal{Z}_{D}$, and $(u_{1}^{\star}, u_{2}^{\star})$ be the minimizer  of \eqref{pbcontrol1-2}-\eqref{defJ2-2} over $\mathcal{Z}_{ad}$.  Let $\psi^\star$ be the associated unique very-weak solution of \eqref{Eq:main_A1} with boundary datum $u_1^\star$, and exterior datum $u_2^\star$.
Then, there exists $p^{\star}=p^\star(u_1^\star,u_2^\star) $ such that the triplet $(\psi^{\star},p^{\star},(u_{1}^{\star}, u_{2}^{\star}))\in L^2((0,T)\times \R^N)\times \mathbb{U}\times  \mathcal{Z}_{ad}$ satisfies the following optimality systems:
\begin{equation}\label{optrho}
\begin{cases}
\psi^{\star}_{t} + \mathscr{L}\psi^{\star} = 0 & \mbox{ in }\; Q,\\
\psi^{\star} = u_{1}^{\star} &\mbox{ in }\; \Gamma , \\
\psi^{\star} = u_{2}^{\star} &\mbox{ in }\; \Sigma , \\
\psi^{\star}(\cdot,0) = 0&\mbox{ in }\; \Omega,
\end{cases}
\end{equation}
and
\begin{equation}\label{optp}
\begin{cases}
-p^{\star}_t + \mathscr{L} p^{\star} = 0 &\mbox{ in } Q, \\
p^{\star} = 0 &\mbox{ in }\; \Sigma , \\
p^{\star}(\cdot,T)  = (-\Delta_D)^{-1}[\psi^{\star}(\cdot,T)- z_d^2] &\mbox{ in }\; \Om,
\end{cases}
\end{equation}
and for all $(v_1,v_2)\in \mathcal{Z}_{ad}$, we have
\begin{equation}\label{opu}
 \int_\Gamma\Big(\partial _\nu p^{\star}-\beta u_{1}^{\star}\Big)(u_{1}^{\star}-v_1)\, \mathrm{d}\sigma\; \dt+
\int_\Sigma\Big(\mathcal{N}_s p^{\star}-\beta u_{2}^{\star}\Big)(u_{2}^{\star}-v_2)\, \dx\, \dt\geq 0 .
 \end{equation}
 In addition,
\begin{equation}\label{coroopu}
( u_{1}^{\star}, u_{2}^{\star})=\mathbb P ( -\beta^{-1}\partial_\nu p^{\star}, -\beta^{-1}\mathcal{N}_s p^{\star}),
 \end{equation}
 where $\mathbb P$ denotes the projection onto the set $\mathcal{Z}_{ad}$.
\end{theorem}	

\begin{proof}
 It follows from the proof of Proposition \ref{thm:docexist} that
$\psi^{\star}$ is the unique very-weak solution of  \eqref{optrho} associated with the minimizer $(u_1^\star,u_2^\star)$.  As above, some classical duality arguments show that \eqref{optp} is the associated dual system. In addition,   using the change of variable $t\mapsto T-t$, we have that $p^{\star}$,  solution of \eqref{optp},  satisfies \eqref{p1} with
$f\coloneqq 0$ and $\phi_0\coloneqq (-\Delta_D)^{-1}[\psi^{\star}(\cdot,T)- z_d^2]\in H^1_0(\Omega)$. Thus,  we can deduce from Theorem \ref{Thm1} that $p^{\star}\in \mathbb{U}. $

To prove the last assertion  \eqref{opu}, we write the Euler Lagrange first order optimality conditions that characterize the optimal control $(u_{1}^\star, u_{2}^\star)$ as follows:
\begin{align*}
\lim_{\lambda\to 0}\frac{J_2(u_1^{\star}+\lambda (v_1-u_1^{\star}),u_2^{\star}+\lambda (v_2-u_2^{\star}))-J_2(u_1^{\star},u_2^{\star})}{\lambda}\geq 0,\;\;\forall (v_1,v_2)\in \mathcal{Z}_{ad}.
\end{align*}
After some calculations, and proceeding as in the proof of Theorem \ref{Thm6}, we obtain that
\begin{align}\label{euler1}
&\langle \psi^\star(\cdot,T),(-\Delta_D)^{-1}[\phi^{\star}(\cdot,T)-z_d^2]
\rangle_{H^{-1}(\Omega),H^1_0(\Omega)}\notag\\
&+\dis
\beta\int_\Gamma u_{1}^{\star}(v_1-u_{1}^{\star})\, d\sigma\, \dt+
\dis \beta\int_\Sigma u_{2}^{\star}(v_2-u_{2}^{\star})\, \dx\, \dt\geq 0 \;\;\forall (v_1,v_2)\in \mathcal{Z}_{ad},
\end{align}
where $\phi^\star= \phi^\star(v_1-u_{1}^{\star},v_2-u_{2}^{\star})$ is the unique very-weak solution of the system
\begin{equation}\label{solz}
\begin{cases}
\phi_t^\star+\mathscr{L} \phi^\star = 0 & \mbox{in } Q, \\
\phi^\star =v_1- u_{1}^{\star} & \mbox{in } \Gamma ,  \\
\phi^\star =v_2-u_{2}^{\star}  & \mbox{in } \Sigma ,  \\
\phi(\cdot,0) = 0 & \mbox{in } \Om.
\end{cases}
 \end{equation}
Recall that under the assumption $0<s\le 3/4$, we have shown in Section \ref{sec-32} that $\partial_\nu p^\star\in  L^2(\Gamma)$ and $\mathcal{N}_s p^{\star}\in L^2(\Sigma)$. In addition, we have that $(-\Delta_D)^{-1}[\psi^{\star}(\cdot,T)- z_d^2]\in H_0^1(\Omega)$.
So,  taking $p^{\star}$ as a test function in the definition of very-weak solutions of  \eqref{solz},   we obtain
\begin{align}\label{euler2}
0=& \langle \psi^\star(\cdot,T),v_2-u_2^{\star})),(-\Delta_D)^{-1}[\phi^{\star}(\cdot,T)-z_d^2]
\rangle_{H^{-1}(\Omega),H^1_0(\Omega)}\notag\\
&+\dis \int_\Gamma(v_1-u_{1}^{\star})\partial_\nu p^{\star}\;\mathrm{d}\sigma\, \dt+\int_\Sigma (v_2-u_{2}^{\star})\mathcal{N}_s p^{\star} \,\dx\, \dt.
\end{align}
Combining \eqref{euler1}-\eqref{euler2}, we get \eqref{opu}.  Here also, the justification of \eqref{coroopu} is classical. The proof is finished.
\end{proof}

\bibliographystyle{plain}

\bibliography{DGW_1}

\end{document}